\documentclass{amsart}
\usepackage{amsmath,amssymb,amsthm,amsfonts,enumerate,array,color,lscape,fancyhdr,layout,pst-all,wasysym}
\usepackage[all,cmtip]{xy}
\usepackage[english]{babel}
\usepackage[latin1]{inputenc}
\usepackage{young}
\usepackage{tikz}

\usepackage[hcentermath]{youngtab}
\usepackage{graphicx}
\usepackage{float}
\usepackage{pstricks}

\usepackage{stmaryrd}

\usepackage{amsthm}

\newcounter{numberofremark}
\newcommand\nothing[1]{}
\usepackage[active]{srcltx}
\usepackage[hyperindex,bookmarksnumbered,plainpages]{hyperref}

\newcommand{\dcl}{\DeclareMathOperator}
\dcl\cdet{cdet} \dcl\Sp{Specm} \dcl\depth{depth} \dcl\im{Im} \dcl\Span{span} \dcl\Ker{Ker} \dcl\Specm{Specm}
\dcl\Supp{Supp} \dcl\codim{codim} \dcl\Y{Y} \dcl\gl{\mathfrak{gl}}    \dcl\U{U} \dcl\T{T}
\dcl\qdet{qdet} \dcl\sgn{sgn} \dcl\gr{gr} \dcl\diag{diag}
\dcl\g{\mathfrak{g}} \dcl\C{\mathbb C} \dcl\dd{{\mathrm d}}
\dcl\ad{{\mathrm ad}}

\newcommand{\mat}[1]{\mbox{\boldmath{$#1$}}}
\newcommand{\smat}[1]{\mbox{\tiny{\boldmath{$#1$}}}}

\newlength\yStones
\newlength\xStones
\newlength\xxStones

\makeatletter
\def\Stones{\pst@object{Stones}}
\def\Stones@i#1{%
  \pst@killglue%
  \begingroup%
  \use@par%
  \setlength\xxStones{\xStones}%
  \expandafter\Stones@ii#1,,\@nil
  \endgroup
  \global\addtolength\xStones{0.6cm}%
  \global\addtolength\yStones{-7.5mm}}%
\def\Stones@ii#1,#2,#3\@nil{%
  \rput(\xxStones,\yStones){%
    \psframebox[framesep=0]{%
      \parbox[c][6mm][c]{11mm}{\makebox[11mm]{$#1$}}}}%
  \addtolength\xxStones{1.2cm}%
  \ifx\relax#2\relax\else\Stones@ii#2,#3\@nil\fi}
\makeatother

\def\Stone#1{\fbox{\makebox[13mm]{\strut#1}}\kern2pt}

\newtheorem{thm}{Theorem}

\newtheorem{theorem}{Theorem}[section]
\newtheorem{lemma}[theorem]{Lemma}
\newtheorem{corollary}[theorem]{Corollary}
\newtheorem{proposition}[theorem]{Proposition}
\newtheorem{example}[theorem]{Example}
\newtheorem{remark}[theorem]{Remark}

\newtheorem{definition}[theorem]{Definition}

\title[Explicit realization of bounded modules]{Explicit realization of bounded modules for symplectic Lie algebras: spinor versus oscillator}
\author[V.Futorny]{Vyacheslav Futorny}
\address{Shenzhen International Center for Mathematics, Southern University of Science and Technology} \email{futorny@sustech.edu.cn} 
\author[D.Gratcharov]{Dimitar Grantcharov}
\address{\noindent
University of Texas at Arlington, Arlington, TX 76019, USA} \email{grandim@uta.edu}
\author[L.E.Ramirez]{Luis Enrique Ramirez}
\address{Universidade Federal do ABC, Santo Andr\'e-SP, Brasil}\email{luis.enrique@ufabc.edu.br}
\author[P.Zadunaisky]{Pablo Zadunaisky}
\address{Instituto de Investigaciones Matem\'aticas Luis A. Santal\'o (IMAS), Buenos Aires, Argentina} 
\email{pzadunaisky@conicet.gov.ar}
\begin{document}
\maketitle
\begin{abstract}
We provide an explicit combinatorial realization of all simple and injective (hence, and projective) modules in the category of bounded $\mathfrak{sp}(2n)$-modules. This realization is defined via a natural tableaux correspondence between spinor-type modules of $\mathfrak{so}(2n)$ and oscillator-type modules of $\mathfrak{sp}(2n)$. In particular, we show that, in contrast with the  $A$-type case, the generic and  bounded $\mathfrak{sp}(2n)$-modules admit an analog of the Gelfand-Graev continuation from finite-dimensional representations.


\end{abstract}

\section*{Introduction}

In the recent years there has been a growing interest in the theory of Gelfand-Tsetlin modules \cite{DFO94}, especially after the discovery of their connection with diagrammatic KLRW algebras \cite{Web17} and Coulomb branches \cite{Web20}. After the breakthrough results in \cite{FGR16}, \cite{FGR17b}, \cite{FGR17a}, \cite{Zad17}, \cite{RZ18}, \cite{Vis18}, \cite{EMV20}, \cite{FGRZ20} and \cite{Web20}, 
 a parameterization of all simple Gelfand-Tsetlin $\mathfrak{gl}(n)$-modules was established in terms of Gelfand-Tsetlin tableaux. One can say that the Gelfand-Tsetlin modules form the largest category of weight modules where such a parameterization is obtained.  The Gelfand-Tsetlin tableaux provide a combinatorial basis on which the action of the generators of  $\mathfrak{gl}(n)$ can be explicitly derived. Unlike for the finite-dimensional modules, this  action  is quite complicated, due to the existence of the so-called ``derivative tableaux''. Let us say that a Gelfand-Tsetlin module $M$ satisfies the \emph{Gelfand-Graev continuation principle} \cite{GG65} if $M$ has a basis of Gelfand-Tsetlin tableaux on which the  action of the generators is expressed in terms of  the classical Gelfand-Tsetlin formulas \cite{GT50}.  In this case, each basis tableau is an eigenvector of the Gelfand-Tsetlin subalgebra -  certain maximal commutative subalgebra of the universal enveloping algebra. In the vertex algebra theory such representations are called \emph{strongly tame} 
 \cite{FKM23}, \cite{FMR21}.

One of our  aims in this paper is to initiate a ``Gelfand-Tsetlin theory'' in the symplectic case. Since the concepts of  
Gelfand-Tsetlin modules is not settled in this case yet, as a first step we look at $\mathfrak{sp}(2n)$-modules that  have Gelfand-Tsetlin tableaux realization.  The tableaux realization of all simple finite-dimensional $\mathfrak{sp}(2n)$-modules was obtained by A. Molev in  \cite{Mol99}, \cite{Mol00a}, and \cite{Mol00b}, and relies on the theory of Yangians. The $C$-type case is particularly complicated because the union of the centers of $U(\mathfrak{sp}(2k))$, $1\leq k \leq n$, is not a maximal subalgebra. In this case one needs to add ``intermediate'' elements to define a Gelfand-Tsetlin type subalgebra $\mathcal{GT}(\mathfrak{sp}(2n))$ of $U(\mathfrak{sp}(2n))$, \cite{MY19}. It is proven that the action of $\mathcal{GT}(\mathfrak{sp}(2n))$ on any simple finite-dimensional $\mathfrak{sp}(2n)$-module $M$ is diagonalizable and has simple spectrum on $M$, see Example 5.6 in \cite{MY19}. However, in the contrast with the $A$-type case, explicit formulas for the action of the generators of $\mathcal{GT}(\mathfrak{sp}(2n))$ on the the tableaux basis of $M$ are not known yet. 

In this paper we consider only modules for which the action of the generators satisfy the  Gelfand-Graev continuation principle. 
We provide an explicit tableaux realization of two types of $\mathfrak{sp}(2n)$-modules: bounded modules (i.e., those weight modules that have uniformly bounded set of weight multiplicities) and generic modules. More precisely, such tableaux basis is given for all simple and indecomposable injective (and projective) modules in the category $\mathcal B$ of all infinite-dimensional bounded $\mathfrak{sp}(2n)$-modules. The simple objects of $\mathcal B$ were classified in \cite{Mat00}, while the injectives (equivalently, the projectives) of  $\mathcal B$ were described in \cite{GS06}. The classification and explicit description of all bounded modules of degree $1$ was obtained in \cite{BL87}.

One important aspect of the paper is that we provide a transparent correspondence between half-integer (i.e., spinor-type) finite-dimensional $\mathfrak{so}(2n)$-modules and infinite-dimensional bounded (i.e., oscillator-type) $\mathfrak{sp}(2n)$-modules. This correspondence was well-understood in the case of the classical spinor and Shale-Weil representations, i.e., the cases involving modules whose weight multiplicities equal to $1$.

We hope that our results are a part of a bigger picture. Namely, that if $M$ is an $\mathfrak{so}(2n)$-module with a basis consisting of $C$-regular tableaux, then for a generic choice of $\mat{\mu} \in \mathbb C^n$, we have an $\mathfrak{sp}(2n)$-module $V(\mat{\mu}, M)$ that has a tableaux realization. The cases covered in this paper include generic $\mathfrak{so}(2n)$-modules $M$ and spinor-type finite-dimensional $\mathfrak{so}(2n)$-modules $M$,  leading to generic and bounded $\mathfrak{sp}(2n)$-modules, respectively. 

In summary, the goal of this paper is threefold - to develop a Gelfand-Tsetlin theory for $\mathfrak{sp}(2n)$-modules that have Gelfand-Graev continuation, to provide an explicit tableaux realization of all bounded modules, and to shed light over the mysterious connection between spinor-type $\mathfrak{so}(2n)$-modules and oscillator-type  $\mathfrak{sp}(2n)$-modules. It is worth noting that in addition to the symplectic case we obtain a Gelfand-Graev continuation for the generic $\mathfrak{so}(2n)$-modules.

The content of the paper is as follows. In Section 2 we collect some important results on tableaux realizations of finite-dimensional modules over Lie algebras of type $C$ and $D$. In Section 3 we provide a summary of the results on bounded $\mathfrak{sp}(2n)$-modules that are used in the paper. In Section 4 we prove all our main results on the tableaux realizations of generic and bounded modules of $\mathfrak{sp}(2n)$.  

For reader's convenience we briefly state some of our results here (see Theorem \ref{module structure}, Theorem \ref{number of D-standard tableaux}, Theorem \ref{thm-main} for the complete versions). Let
$\mat{\lambda}\in (\frac{1}{2}+\mathbb{Z})^n$ be a dominant  weight of $\mathfrak{so}(2n)$ and  $\mat{\mu}\in(\mathbb{C}\setminus\mathbb{Z})^n$. Denote by $\mathcal{B}(\mat{\mu},\mat{\lambda})$ the set of tableaux $T(W)$ of type $C$ with top row $\mat{\lambda}+\rho_{D}+\mat{\frac{1}{2}}=(\lambda_{1}-1+\frac{3}{2},\ldots,\lambda_{n}-n+\frac{3}{2})$, whose left-most column is in 
$\mat{\mu}+\mathbb{Z}^{n}$, and whose remaining part is a $D$-standard tableau
 (see Definitions \ref{def-T_D}, \ref{V space}). Denote by $V(\mat{\mu},\mat{\lambda})$
 the $\mathbb{C}$-vector space  with basis $\mathcal{B}(\mat{\mu},\mat{\lambda})$.

\begin{thm}
The space $V(\mat{\mu}, \mat{\lambda})$ has a structure of an $\mathfrak{sp}(2n)$-module obtained by the Gelfand-Graev continuation.
\end{thm}

An important property of $V(\mat{\mu}, \mat{\lambda})$ is that its simple subquotients $V(\mat{\mu}, \mat{\lambda}, \Sigma)$ (see Definition \ref{def-Sigma}) exhaust all simple bounded $\mathfrak{sp}(2n)$-modules. The  subquotients $V(\mat{\mu}, \mat{\lambda}, \Sigma)$  are parameterized by subsets 
$\Sigma$ of ${\rm Int} (2{\mat{\mu}})$, where ${\rm Int}(\mat{z})$ stands for  the set of all $i \in \llbracket n \rrbracket $ such that $z_i \in \mathbb Z$. 

\begin{thm}
\begin{itemize}
\item[(i)] Every infinite-dimensional simple bounded $\mathfrak{sp}(2n)$-module is isomorphic to $V(\mat{\mu}, \mat{\lambda}, \Sigma)$ for some $\mat{\mu},\mat{\lambda}, \Sigma$ as above.
\item[(ii)] The module $V(\mat{\mu}, \mat{\lambda})$ is the injective envelope of $V(\mat{\mu}, \mat{\lambda},{\rm Int} (2{\mat{\mu}}))$ in the category of bounded $\mathfrak{sp}(2n)$-modules. 
\end{itemize}
In particular, we have an explicit tableaux realization of all simple and indecomposable injective (and projective) modules in the category of bounded $\mathfrak{sp}(2n)$-modules.
\end{thm}

\medskip
\noindent{\bf Acknowledgements.} 
V.F. is supported in part by China (NSFC) - 12350710178. D.G. is supported in part by Simons Collaboration 
Grant 855678. L.E.R. is supported in part by CNPq Grant 316163/2021-0, and Fapesp grant 2022/08948-2.
 P.Z. is a CONICET researcher.

\section{Notation and conventions}
Throughout this paper we fix $n\geq 2$. Given two complex numbers $a,b$, we write $a\geq b$ (respectively $a>b$) if $a-b\in\mathbb{Z}_{\geq 0}$ (respectively $a-b\in\mathbb{Z}_{>0}$). For $m\in \mathbb{Z}_{>0}$ we set $\llbracket  m \rrbracket =\{1,2,...,m \}$. We often  write $\mat{a}$ for the $n$-tuple $(a,a,\ldots,a)$ when $a$ is fixed (for example, $a=1$ and $a=1/2$). For $\mat{z}=(z_1,...,z_n) \in \mathbb C^n$, we write $|z|$ for the sum $\sum_{i=1}^n|z_i|$, and  ${\rm Int}(\mat{z})$ for  the set of all $i \in \llbracket n \rrbracket $ such that $z_i \in \mathbb Z$. We also set ${\rm Hf} (\mat{z}) = {\rm Int}(\mat{z + 1/2})$.

\section{Tableaux realization of simple finite-dimensional modules over algebras of type $C$ and $D$}
Recall that simple finite-dimensional modules for the Lie algebras of type $C
_n$ and $D_n$ are in one-to-one correspondence with $n$-tuples  $\mat{\lambda}=(\lambda_1, \ldots,\lambda_n)$ in $\left(\frac{1}{2}+\mathbb{Z}\right)^{n}\bigcup\mathbb{Z}^{n}$ where the numbers $\lambda_i$ are such that:
\begin{align}\label{sp-dominant}\lambda_i \geq  \lambda_{i+1}\text{ for }i\in \llbracket n-1 \rrbracket ,\text{ and }-\lambda_1\in\mathbb{Z}_{\geq 0},\text{ for }\mathfrak{sp}(2n).\\\label{so-dominant}
\lambda_i \geq  \lambda_{i+1}\text{ for }i\in \llbracket n-1 \rrbracket ,\text{ and }-\lambda_1 -\lambda_2\in\mathbb{Z}_{\geq 0},\text{ for }\mathfrak{so}(2n).
\end{align}

A weight $\mat{\lambda}$ satisfying (\ref{sp-dominant}) will be called \emph{dominant $\mathfrak{sp}(2n)$-weight}, and the corresponding simple finite-dimensional module will be denoted by $L_{C}(\mat{\lambda})$. Analogously, a weight $\mat{\lambda}$ satisfying (\ref{so-dominant}) will be called \emph{dominant $\mathfrak{so}(2n)$-weight}, and the corresponding simple finite-dimensional module will be denoted by $L_{D}(\mat{\lambda})$. The simple roots of $\mathfrak{sp}(2n)$ that lead to the dominance conditions \eqref{sp-dominant} are listed in  \S \ref{subsec-cn}, and the simple roots of $\mathfrak{so}(2n)$ related to \eqref{so-dominant} are described in a similar manner. With these choices of simple roots, we have that the corresponding half-sums of positive roots are
$$
\rho_{C} = (-1,-2,...,-n); \; \rho_{D} = (0,-1,...,-n+1).
$$

\subsection{Type $D$}
In this subsection we recall the tableaux realization of the simple finite-dimensional $\mathfrak{so}(2n)$-modules.
\begin{definition}
An array of $n^2$ complex numbers arranged in $2n-1$ rows in the following form 
$$
{\footnotesize T(L):=\begin{tabular}{cccccccccc}
\xymatrixrowsep{0.2cm}
\xymatrixcolsep{0.1cm}\xymatrix @C=0.5em {&\ell_{n1}& &\ell_{n2}& &\ell_{n3}&&\cdots&&\ell_{nn}\\
& &\ell'_{n,2}& &\ell'_{n,3}&&\cdots&&\ell'_{n,n}&\\
&\ell_{n-1,1}& &\ell_{n-1,2}& &\cdots&&\ell_{n-1,n-1}&&\\
& &\ell'_{n-1,2}& &\cdots&&\ell'_{n-1,n-1}&&&\\
&\vdots& &\vdots&&\vdots&&&&\\
& &\vdots&&\vdots&&&&&\\
&\ell_{21}& &\ell_{22}& &&&&&\\
& &\ell'_{22}& &&&&&&\\
&\ell_{11}& && &&&&&\\
}\end{tabular} }
$$
will be called  \emph{Gelfand-Tsetlin tableau (or just tableau) of type $D$}. 
\end{definition}
\begin{remark}
Note that for the sake in convenience in the definition above, the entries $\ell'_{ij}$ appear for $2\leq j\leq i\leq n$, instead of the more standard convention when they appear for  $1\leq j\leq i\leq n-1$ (see for example  \S 9.6 in \cite{Mol07}, and  \S 3 in \cite{Mol00a}). 
\end{remark}

 \begin{definition}\label{D-standard}
A tableau $T(U)$ of type $D$ is called \emph{$D$-standard} if for any $2\leq k\leq n$, the following conditions are satisfied:
\begin{align}\label{condition 5}
&-u'_{k,2}\geq u_{k,1}>u'_{k2}\geq u_{k2}> \cdots > u'_{k,k-1}\geq u_{k,k-1}> u'_{k,k}\geq u_{k,k}.\\\label{condition 6}
&-u'_{k,2}\geq u_{k-1,1}>u'_{k,2}\geq u_{k-1,2}> \cdots> u'_{k,k-1}\geq u_{k-1,k-1}> u'_{k,k}.
\end{align}
\end{definition}

 \begin{remark}\label{D-standard vs patterns} In \cite{Mol00a} a slight modification of the above conditions is used. Namely, the strict inequalities are replaced by nonstrict and the corresponding tableau is referred as a $D$-type pattern. There is a bijection between the $D$-stardard tabelaux and the $D$-type patterns defined by $\ell_{ij}\to \ell_{ij}-j+\frac{3}{2}$, $\ell'_{ij}\to \ell'_{ij}-j+\frac{3}{2}$.
 \end{remark}

\begin{theorem}[\cite{Mol07}, Lemma 9.6.7.]\label{shifted D}If $\mat{\lambda}=(\lambda_1,\ldots,\lambda_n)$ is a dominant $\mathfrak{so}(2n)$-weight, the simple finite-dimensional module $L_{D}(\mat{\lambda})$ admits a tableaux realization with basis given by the set of $D$-standard tableaux $T(W)$ with  top row $\mat{\lambda}+ \rho_D + \mat{1/2}$. 
\end{theorem}

\subsection{Type $C$} \label{subsec-cn}
Consider $I:=\{1,\ldots,n,-1,\ldots,-n\}\times \{1,\ldots,n,-1,\ldots,-n\}$, and square matrices of size $2n$ indexed by $I$. Also, for $(i,j)\in I$ consider $F_{ij}:=E_{ij}-\sgn(i)\sgn(j)E_{-j,-i}$, where $E_{ij}$ denotes the elementary matrix with $1$ at $(i,j)$-th position and $0$ otherwise. As a Lie algebra, $\mathfrak{sp}(2n)$ is generated by the elements $F_{-k,k}$,  $F_{k,-k}$ with $k\in\llbracket n \rrbracket $ and $F_{k-1,-k}$ with $k=2,\ldots,n$. We fix   $\mathfrak{h}:=\Span \{ F_{11}, \ldots, F_{nn}\}$ to be the Cartan subalgebra of $\mathfrak{sp}(2n)$, and denote by $\Delta = \Delta (\mathfrak{sp}(2n), \mathfrak h)$ the corresponding root system and $\mathcal Q_C = \mathbb Z \Delta$ the corresponding root lattice. By $\{\varepsilon_1,...,\varepsilon_n \}$ we denote the basis of $\mathfrak h^*$ dual to $\{ F_{11},...,F_{nn}\}$.  We will often identify $\mathfrak h^*$ with $\mathbb C^n$ via the correspondence $\sum_{i=1}^n \lambda_i \varepsilon_i \mapsto (\lambda_1,...,\lambda_n)$ and will denote both elements by $\mat{\lambda}$.

In explicit terms we have $\Delta = \{ \pm 2 \varepsilon_i, \pm \varepsilon_k \pm \varepsilon_l \; | \; k\neq l\}$. Also, the set of simple roots that lead to the dominance conditions \eqref{sp-dominant} is $\Pi = \{-2\varepsilon_1, \varepsilon_1-\varepsilon_2,...,\varepsilon_{n-1}-\varepsilon_n\}$, in particular,  $\rho_C = (-1,...,-n)$. The reflections of the Weyl group corresponding to the long roots $2\varepsilon_i$ will be denoted by $\tau_i$. In particlar,  $\tau_{1}(\mat{\lambda})=(-\lambda_1,\lambda_2,\ldots,\lambda_n)$.

\begin{definition}\label{Type C}
A \emph{(Gelfand-Tsetlin) tableau of type $C$} is an array of $n^2+n$ complex numbers arranged in the following form
$$
{\footnotesize T(L):=\begin{tabular}{cccccccccc}
\xymatrixrowsep{0.2cm}
\xymatrixcolsep{0.1cm}\xymatrix @C=0.5em {
&\ell_{n1}& &\ell_{n2}& &\cdots&&\cdots&\ell_{nn}\\
\ell'_{n1}& &\ell'_{n2}& &\cdots&&\cdots&\ell'_{nn}&\\
&\ell_{n-1,1}& &\ell_{n-1,2}& &\cdots&\ell_{n-1,n-1}&&\\
\ell'_{n-1,1}& &\ell'_{n-1,2}& &\cdots&\ell'_{n-1,n-1}&&&\\
&\vdots& &\vdots&&\vdots&&&\\
\vdots& &\vdots&&\vdots&&&&\\
&\ell_{21}& &\ell_{22}& &&&&\\
\ell'_{21}& &\ell'_{22}& &&&&&\\
&\ell_{11}& && &&&&\\
\ell'_{11}& && &&&&&\\
}\end{tabular}}
$$

\end{definition}

\begin{definition}\label{C-standard}
A tableau $T(L)$ of type $C$ will be called \emph{$C$-standard} if its entries satisfy the following conditions:
\begin{align}\label{condition 7}
-\frac{1}{2}\geq \ell'_{k1}\geq \ell_{k1}> \ell'_{k2}\geq \ell_{k2}> \cdots > \ell'_{k,k-1}\geq \ell_{k,k-1}> \ell'_{k,k}\geq \ell_{k,k},\  &k\in\llbracket n \rrbracket .\\\label{condition 8}
-\frac{1}{2}\geq \ell'_{k1}\geq \ell_{k-1,1}> \ell'_{k2}\geq \ell_{k-1,2}> \cdots> \ell'_{k,k-1}\geq \ell_{k-1,k-1}> \ell'_{kk},&\ 2\leq k\leq n
\end{align}

\end{definition}

\begin{theorem}[\cite{Mol07}, Theorem 9.6.2]\label{fd sp} Let $\mat{\lambda}=(\lambda_{1},\ldots,\lambda_{n})$ be a dominant $\mathfrak{sp}(2n)$-weight. The simple finite-dimensional module $L_C(\mat{\lambda})$ admits a basis parameterized by the set of all $C$-standard tableau $T(L)$  with top row $ \mat{\lambda} + \rho_{C} + \mat{1/2}$. Moreover, the action of the generators of $\mathfrak{sp}(2n)$ on the basis elements is given by:
\end{theorem}
\begin{align}\label{Formulas sp}
F_{kk}T(L)=&\left(2\sum_{i=1}^{k}\ell'_{ki}-\sum_{i=1}^{k}\ell_{ki}-\sum_{i=1}^{k-1}\ell_{k-1,i}+k-\frac{1}{2}\right)T(L),\\\label{Formulas sp1}
F_{k,-k}T(L)=&\sum\limits_{i=1}^{k}A_{ki}(L)T(L+\delta^{(ki)'}),\\\label{Formulas sp2}
F_{-k,k}T(L)=&\sum\limits_{i=1}^{k}B_{ki}(L)T(L-\delta^{(ki)'}),\\\label{Formulas sp3}
F_{k-1,-k}T(L)=&\sum\limits_{i=1}^{k-1}C_{ki}(L)T(L-\delta^{(k-1,i)})\\
&+\sum\limits_{i=1}^{k}\sum\limits_{j,m=1}^{k-1}D_{kijm}(L)T(L+\delta^{(ki)'}+\delta^{(k-1,j)}+\delta^{(k-1,m)'}),\\\label{Formulas sp4}
2F_{k,k-1}T(L)=&\ \big[F_{k-1,-k}, F_{1-k,k-1}\big]T(L),\\\label{Formulas sp5}
2F_{k-1,k}T(L)=&\ \big[F_{k-1,-k}, F_{-k,k}\big]T(L),
\end{align}
 where
\begin{align}\label{operator A}
A_{ki}(L)&=\prod\limits_{a=1,a\neq i}^{k}\frac{1}{\ell'_{ka}-\ell'_{ki}},\\\label{operator B}
B_{ki}(L)&=2A_{ki}(L)(2\ell'_{ki}-1)\prod\limits_{a=1}^{k}(\ell_{ka}-\ell'_{ki})\prod\limits_{a=1}^{k-1}(\ell_{k-1,a}-\ell'_{ki}),\\\label{operator C}
C_{ki}(L)&=\frac{1}{(2\ell_{k-1,i}-1)}\prod\limits_{a=1,a\neq i}^{k-1}\frac{1}{(\ell_{k-1,i}-\ell_{k-1,a})(\ell_{k-1,i}+\ell_{k-1,a}-1)},\\\label{operator D}
D_{kijm}(L)&=A_{ki}(L)A_{k-1,m}(L)C_{kj}(L)\prod\limits_{a\neq i}^{k}(\ell^2_{k-1,j}-\ell^{'2}_{k,a})\prod\limits_{a\neq m}^{k-1}(\ell^2_{k-1,j}-\ell^{'2}_{k-1,a}).
\end{align}

\begin{remark} \label{rem-thm-d}
There is an analogous to the theorem above in the $D$-type case. Recall from Theorem \ref{shifted D} that for a dominant $\mathfrak{so}(2n)$-weight $\mat{\eta} \in \mathbb C^n$, $L_D(\mat{\eta})$ admits a basis parameterized by the set of all $D$-standard tableau $T(L)$  with top row $ \mat{\eta} + \rho_{D} + \mat{1/2}$. The action of the generators of $\mathfrak{so}(2n)$ on these $T(L)$ can  be written explicitly and involve some rational functions of $\mat{\eta}$. To keep this paper brief, we will not include these formulas and will refer to them as \emph{GT-formulas of type $D$}. We refer the reader to Theorem 9.6.8 in \cite{Mol07} for the missing details.
\end{remark}

\begin{definition}\label{set of standard tableaux}
Let $\mat{\lambda}\in \mathbb C^n$ be a dominant $\mathfrak{sp}(2n)$-weight, and $\mat{\eta}\in \mathbb C^n$ be a dominant $\mathfrak{so}(2n)$-weight. We denote by $C_{st}^{\mat{\lambda}}$ the set of all $C$-standard tableaux with top row $\mat{\lambda}+ \rho_C + \mat{1/2}$, and  by $D_{st}^{\mat{\eta}}$ the set of all $D$-standard tableaux with top row $\mat{\eta}+ \rho_D + \mat{1/2}$.
\end{definition}

\begin{remark}\label{center is polynomial in lambda}
Given  an $\mathfrak{sp}(2n)$-weight $\mat{\lambda}=(\lambda_{1},\ldots,\lambda_{n})$, any element $z$ of the center $Z(U(\mathfrak{sp}(2n)))$ of $U(\mathfrak{sp}(2n))$  acts on $L_{C}(\mat{\lambda})$ via a multiplication by a scalar. By the Harish-Chandra Theorem, this scalar is a polynomial $p_{z}(\mat{\lambda})$ in $\lambda_{1},\ldots,\lambda_{n}$. More generally, for a simple $\mathfrak{sp}(2n)$-module $M$, there is $\mat{\lambda} \in \mathbb C^n$ so that $z$ acts on $M$ via $p_{z}(\mat{\lambda})$. In the latter case we call $\chi_{\smat{\lambda}\mat{}} : Z(U(\mathfrak{sp}(2n)))  \to \mathbb C$, $\chi_{\smat{\lambda}\mat{}} (z) = p_{z}(\mat{\lambda})$, the central character of $M$ (and of $L_C(\mat{\lambda})$). The same result holds in the case of $\mathfrak{so}(2n)$, and we will use the same notation (i.e., $p_{z}$ for the corresponding polynomial). 
\end{remark}

 We will make use of the following result, the proof of which relies on the fact that the set $\bigcup_{\mat{\lambda}} C_{st}^{\mat{\lambda}}$ (union taken over all $\mathfrak{sp}(2n)$-dominant $\mat{\lambda} \in \mathbb C^n$) is Zariski dense in $\mathbb C^{n^2+n}$. For a rigorous proof of a more general result in the case of $\mathfrak{gl}(n)$, see Lemma 3.4 in \cite{Zad17}. Henceforth, for convenience we will write $f(W)$ for a polynomial $f$ in $n^2+n$ variables evaluated at the entries of a tableau $T(W)$ of type $C$.
 
\begin{lemma}\label{density lemma} Let $f(\mat{x})$ be a polynomial in $n^2+n$ variables such that for any dominant $\mathfrak{sp}(2n)$-weight $\mat{\lambda}$, $f(W)=0$ 
for all $T(W)\in C_{st}^{\mat{\lambda}}$. Then $f=0$. 
\end{lemma}

\begin{example}
In the following, given $a,b\in \mathbb C$, we use an arrow from $a$ to $b$ if $a-b\in\mathbb{Z}_{\geq 0}$, and a dotted arrow from $a$ to $b$ if $a-b\in\mathbb{Z}_{>0}$. With this convention, for $n=4$, the Relations (\ref{condition 5}, \ref{condition 6}, \ref{condition 7}, \ref{condition 8}) can be described by directed graphs as follows:

\medskip

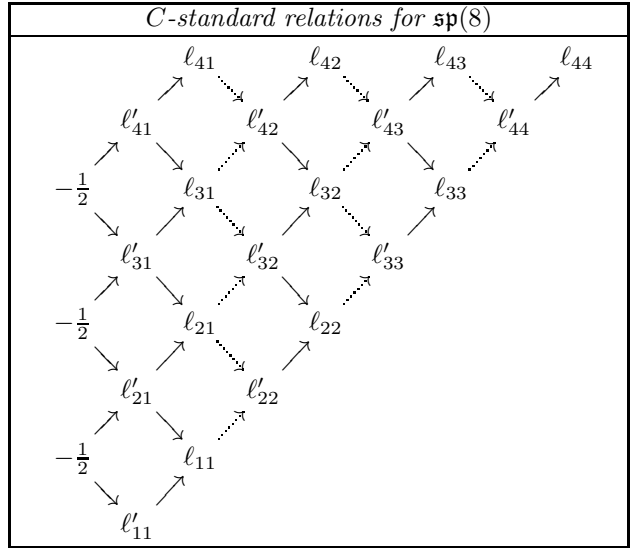
\begin{center}

\begin{tabular}{|c|}
\hline  $C$-standard relations for $\mathfrak{sp}(8)$\\ \hline
$\begin{tabular}{cccccccccc}
\xymatrixrowsep{0.3cm}
\xymatrixcolsep{0.1cm}\xymatrix @C=0.5em {
&&\ell_{41}\ar@{.>}[rd]& &\ell_{42}\ar@{.>}[rd]&&\ell_{43}\ar@{.>}[rd]&&\ell_{44}\\
&\ell'_{41}\ar[rd]\ar[ru]& &\ell'_{42}\ar[rd]\ar[ru]&&\ell'_{43}\ar[rd]\ar[ru]&&\ell'_{44}\ar[ru]&\\
 -\frac{1}{2}\ar[rd]\ar[ru]&&\ell_{31}\ar@{.>}[rd]\ar@{.>}[ru]& &\ell_{32}\ar@{.>}[rd]\ar@{.>}[ru]&&\ell_{33}\ar@{.>}[ru]&&\\
& \ell'_{31}\ar[rd]\ar[ru]& &\ell'_{32}\ar[rd]\ar[ru]&&\ell'_{33}\ar[ru]&&&\\
 -\frac{1}{2}\ar[rd]\ar[ru]&&\ell_{21}\ar@{.>}[rd]\ar@{.>}[ru]& &\ell_{22}\ar@{.>}[ru]& &&&\\
& \ell'_{21}\ar[rd]\ar[ru]& &\ell'_{22}\ar[ru]& &&&&\\
 -\frac{1}{2}\ar[rd]\ar[ru]&&\ell_{11}\ar@{.>}[ru]& && &&&\\
&\ell'_{11}\ar[ru]&& && &&&\\
}\end{tabular}$\\ \hline
 	\end{tabular}
	\end{center}

 \medskip
\begin{center} \hspace{1cm}
	\begin{tabular}{|c|}
\hline $D$-standard relations for $\mathfrak{so}(8)$\\ \hline $\begin{tabular}{cccccccccc}
\xymatrixrowsep{0.3cm}
\xymatrixcolsep{0.1cm}
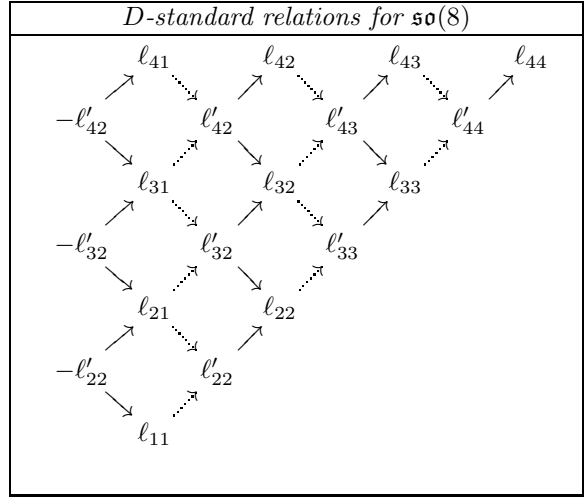
\xymatrix @C=0.5em {
&\ell_{41}\ar@{.>}[rd]& &\ell_{42}\ar@{.>}[rd]&&\ell_{43}\ar@{.>}[rd]&&\ell_{44}\\
 -\ell'_{42}\ar[rd]\ar[ru]& &\ell'_{42}\ar[rd]\ar[ru]&&\ell'_{43}\ar[rd]\ar[ru]&&\ell'_{44}\ar[ru]&\\
&\ell_{31}\ar@{.>}[rd]\ar@{.>}[ru]& &\ell_{32}\ar@{.>}[rd]\ar@{.>}[ru]&&\ell_{33}\ar@{.>}[ru]&&\\
 -\ell'_{32} \ar[rd]\ar[ru]& &\ell'_{32}\ar[rd]\ar[ru]&&\ell'_{33}\ar[ru]&&&\\
&\ell_{21}\ar@{.>}[rd]\ar@{.>}[ru]& &\ell_{22}\ar@{.>}[ru]& &&&\\
-\ell'_{22}\ar[rd]\ar[ru]& &\ell'_{22}\ar[ru]& &&&&\\
&\ell_{11}\ar@{.>}[ru]& && &&&\\
&& && &&&\\
}\end{tabular}$\\ \hline
 	\end{tabular}
\end{center}

\end{example}

\section{Generalities on bounded $\mathfrak{sp}(2n)$-modules}

In this section we collect some important definitions and properties of the bounded modules of $\mathfrak{sp}(2n)$. For rigorous proofs and more details we refer the reader to \cite{Mat00} and \cite{GS06}.

\subsection{Weight and bounded modules} \label{subsec-bounded}
 A \emph{weight  module} $M$ is an $\mathfrak{sp}(2n)$-module that is semisimple as $\mathfrak h$-module. The $\mathfrak h$-isotypic components of $M$ are the \emph{weight spaces} of $M$: we denote them by $M^{\lambda}$ for $\lambda \in \mathfrak h^*$.  Every weight module $M$ has a well-defined support:
       $$ \Supp  M = \{\mu \in \mathfrak h^* \; | \; M^{\mu} \neq 0 \}. $$

A \emph{weight}  module is \emph{bounded} if the dimension of any weight space of $M$ is less or equal to  $N$ for some fixed $N \in \mathbb Z_{\geq 0}$. The \emph{degree} $d (M)$ of a bounded weight module $M$ equals the maximum value of the weight multiplicities  of $M$. The essential support of a bounded module $M$ is 
  $$ \Supp_{\rm ess}  M = \{\mu \in \Supp  M \; | \; \dim M^{\mu} = d(M) \}. $$
For an infinite-dimensional bounded $\mathfrak{sp}(2n)$-module $M$, we say that $M$ has a \emph{cone} $\mathcal C (M) \subset \mathcal Q_C$ if:
$$
\mat{\lambda}' + \mathcal C(M) \subset \mbox{Supp}_{\rm ess}M \subset  \mbox{Supp} M \subset \mat{\lambda}'' + \mathcal C(M)
$$
for some weights $\mat{\lambda}' , \mat{\lambda}''$ in  $\mbox{Supp} M$. 

The infinite-dimensional simple bounded modules have cones and in order to describe these cones, we introduce some notation. We set $\mathcal C_k := \left\{ x \in \mathcal Q_C \; | \; \langle x,\varepsilon_k \rangle \geq 0 \right\} $ and for $\Sigma \subseteq \llbracket n \rrbracket $, we set $\mathcal C_{\Sigma} = \bigcap_{k \in \Sigma} \mathcal C_k$. If $\Sigma,\Sigma'$ are disjoint subsets of $\llbracket n \rrbracket $, we set 
$$ \mathcal C_{\Sigma,\Sigma'} = \mathcal C_{\Sigma}\cap \left(- \mathcal C_{\Sigma'}\right).
$$
Recall that for $\mat{\mu} \in \mathbb C^n$, ${\rm Hf} (\mat{\mu})= {\rm Int} ({\mat{\mu}}+ \mat{1/2})$. Also, for $\overline{\mat{\nu}} = \mat{\nu} + \mathcal Q_C \in \mathbb{C} ^{n} /\mathcal Q_C$ we define ${\rm Hf} (\overline{\mat{\nu}}) = {\rm Hf} (\mat{\nu})$.

\begin{example} \label{sp-2n-sigma}
In this example we provide explicit realization of all simple bounded infinite-dimensional $\mathfrak{sp}(2n)$-modules of central character $\sigma = \chi_{\smat{1/2}\mat{}}$. Every such module is uniquely determined by a coset $\overline{\mat{\nu}} = \mat{\nu} + \mathcal Q_C$ in  $\mathbb C^n/\mathcal Q_C$ and a subset $\Sigma \subset \overline{\mat{\nu}}$.

Let $\mathcal D (n)$ be the n$th$ Weyl algebra, i.e. the algebra of polynomial differential operators on $\mathcal O_n = \mathbb C [t_1,...,t_n]$. Then there is a homomorphism 
$$
\varphi: U(\mathfrak{sp}(2n)) \to \mathcal D (n)
$$
of associative algebras such that $\varphi(F_{-k,k}) = \frac{1}{\sqrt{2}} \partial_{t_k}^2$ and $\varphi(F_{k,-k}) = \frac{1}{\sqrt{2}} t_k^2$, see for example \S 5 in \cite{GS06}. The image of $\varphi$ is $\mathcal D (n)^{\rm ev}$,  the subspace of $\mathcal D (n)$ spanned by $t^{\alpha} \partial_t^{\beta}$ with $|\alpha| + |\beta|$ even. 

The pullback through $\varphi $ of 
the space  $
\mathcal F^{\rm ev} =  \mathbb C[t_1,...,t_n]^{\rm ev}$
of polynomials of even degree (respectively, of the space  $
\mathcal F^{\rm odd} =  \mathbb C[t_1,...,t_n]^{\rm odd}$
of polynomials of odd degree) is a simple bounded $\mathfrak{sp} (2n)$-module of degree $1$, central character $\sigma$, and cone $\mathcal C_{\llbracket n \rrbracket , {\emptyset}}$.
Then for any $\overline{\mat{\nu}} = \mat{\nu} + \mathcal Q_C$ in  $\mathbb C^n/\mathcal Q_C$, the pullback of 
$$
\mathcal F(\overline{\mat{\nu}}) =  \Span \{t^{\smat{\nu} + \smat{z}} \; | \; z_i \in \mathbb Z, i \in {\rm Int} (\mat{\nu}), z_j  \in \mathbb Z_{\geq 0}, j\notin {\rm Int}(\mat{\nu}), z_1+\cdots + 
z_n \in 2\mathbb Z \}$$
is a simple bounded $\mathfrak{sp} (2n)$-module of degree $1$,  cone $\mathcal C_{ {\rm Int }(\overline{\mat{\nu}}), \emptyset}$, $\Supp \mathcal F(\overline{\mat{\nu}}) \subset  \overline{\mat{\nu} + \mat{\frac{1}{2}}}$,   and $\chi_{\mathcal F(\overline{\mat{\nu}})} = \chi_\sigma$. In particular, $\mathcal F^{\rm ev} = \mathcal F (\overline{\mat{0}}) \simeq L_C(\mat{1/2}\mat{})$ and $\mathcal F^{\rm odd} = \mathcal F (\overline{\varepsilon}_1) \simeq L_C(\mat{1/2}\mat{} + \varepsilon_1)$. 

Finally, for a subset  $\Sigma$ of ${\rm Int}(\overline{\mat{\nu}})$, we may twist the module $\mathcal F (\overline{\mat{\nu}})$ by the automorphism $\theta_{\Sigma}$ of $\mathcal D (n)$ defined by $t_i \mapsto \partial_{t_i}$, 
$\partial_{t_i} \mapsto -t_i$, $i \in \Sigma$; 
$t_j \mapsto t_j$,  $\partial_{t_j} \mapsto \partial_{t_j}$, $j \notin \Sigma$, and obtain a bounded module $\mathcal F (\overline{\mat{\nu}}, \Sigma)$ of degree $1$, central character $\sigma$, and cone $\mathcal C_{{\rm Int}(\overline{\mat{\nu}})\setminus \Sigma, \Sigma}$. In particular, $\mathcal F (\overline{\mat{\nu}}, \emptyset) = \mathcal F (\overline{\mat{\nu}})$. The set of all $\mathcal F (\overline{\mat{\nu}}, \Sigma)$ is the complete set of isomorphism classes of simple bounded infinite-dimensional $\mathfrak{sp}(2n)$-modules of central character $\sigma$.
\end{example}

\begin{proposition} \label{prop-class-sp}
Let $\overline{\mat{\nu}}\in \mathbb{C} ^{n} /\mathcal Q_C$,  $\mat{\lambda}\in(\frac{1}{2}+\mathbb{Z})^{n}$ be a dominant $\mathfrak{so}(2n)$-weight, and $\Sigma \subset {\rm Hf} ( \overline{\mat{\nu}})$.
\begin{itemize}
\item[(i) ] There is a unique up to isomorphism simple bounded  $\mathfrak{sp}(2n)$-module $S = S(\chi_{\smat{\lambda} + \smat{1}\mat{}}, \overline{\mat{\nu}}, \Sigma)$ with central character  $\chi_S = \chi_{\smat{\lambda} + \smat{1}\mat{}}$, support $\Supp S \subset\overline{\mat{\nu}} $, and cone $\mathcal C_S = \mathcal C_{{\rm Hf}(\overline{\mat{\nu}}) \setminus \Sigma, \Sigma}$.

\item[(ii)] If $S(\chi_{\smat{\lambda} + \smat{1}\mat{}}, \overline{\mat{\nu}}, \Sigma)$ is the module defined in (i), then every infinite-dimensional simple bounded  $\mathfrak{sp}(2n)$-module is isomorphic to some  module $S(\chi_{\smat{\lambda} + \smat{1}\mat{}}, \overline{\mat{\nu}}, \Sigma)$. In particular, $\mathcal F (\overline{\mat{\nu}}, \Sigma) \simeq S(\chi_{\sigma}, \overline{\mat{\nu} + \mat{1/2}\mat{}}, \Sigma)$.

\item[(iii)] $\deg S(\chi_{\smat{\lambda} + \smat{1}\mat{}}, \overline{\mat{\nu}}, \Sigma)  = \frac{1}{2^{n-1}} \dim L_D(\mat{\lambda}) $.
\end{itemize}
\end{proposition}
\begin{proof} The statements follow from the classification of irreducible semisimple coherent families in \S 12 of \cite{Mat00}, but for reader's convenience, we outline the proof. 

 Let $\chi = \chi_{\smat{\lambda} + \smat{1}\mat{}}$ and $\sigma = \chi_{\smat{1/2}\mat{}}$. The translation functor $\theta_{\sigma}^{\chi} (\bullet) = \mbox{pr}^{\chi} (\bullet \otimes L_C(\mat{\lambda} + \mat{1/2}\mat{}))$ gives an equivalence of categories between the bounded $\mathfrak{sp}(2n)$-modules with  central character $\sigma$ and  the bounded $\mathfrak{sp}(2n)$-modules with  central character $\chi$. Moreover, the translation functor preserves the cone   $\mathcal C_S$ of a simple module $S$. Hence, for the sake of simplicity we may assume that $\chi = \sigma$, i.e. that $\mat{\lambda} = \mat{-1/2}$. But the simple bounded  $\mathfrak{sp}(2n)$-modules with central character $\sigma$ are explicitly described in Example \ref{sp-2n-sigma}, which completes the proof of (i) and (ii). For part (iii), see for example Theorem 12.2 in \cite{Mat00}. \end{proof}

\subsection{Twisted localization and injectives}
We first recall some properties of the twisted localization functor in general. Let $\mathcal U$ be an associative unital algebra and $\mathcal H$ be a commutative subalgebra of $\mathcal U$. We assume in addition that $\mathcal H =  \C [\mathfrak h]$ for some vector space ${\mathfrak h}$, and that 
$$\mathcal U=\bigoplus_{\mu\in {{\mathfrak h}^*}}\mathcal U^\mu,$$
where
$$\mathcal U^\mu=\{x\in\mathcal U\ |\  [h,x]=\mu(h)x, \forall h\in\mathfrak h\}$$
is the $\mu$-weight space of $\mathcal U$.

Let  $a$ be an ad-nilpotent element of $\mathcal U$. Then the set $\langle a \rangle = \{ a^n \; | \; n \geq 0\}$ is an Ore subset of $\mathcal U$  which allows us to define the  $\langle a \rangle$-localization $D_{\langle a \rangle} \mathcal U$ of $\mathcal U$. For a $\mathcal U$-module $M$  by $D_{\langle a \rangle} M = D_{\langle a \rangle} {\mathcal U} \otimes_{\mathcal U} M$ we denote the $\langle a \rangle$-localization of $M$. Note that if $a$ is injective on $M$, then $M$ is isomorphic to a submodule of $D_{\langle a \rangle} M$. In the latter case we will identify $M$ with that submodule.

We next recall the definition of the generalized conjugation of $D_{\langle a \rangle} \mathcal U$ relative to $x \in {\mathbb C}$, see  Lemma 4.3 in \cite{Mat00}. This is the automorphism  $\phi_x : D_{\langle a \rangle} \mathcal U \to D_{\langle a \rangle} \mathcal U$ given by $$\phi_x(u) = \sum_{i\geq 0} \binom{x}{i} \ad (a)^i (u) a^{-i}.$$ If $x \in \mathbb Z$, then $\phi_x(u) = a^xua^{-x}$. With the aid of $\phi_x$ we define the twisted module $\Phi_x(M) = M^{\phi_x}$ of any  $D_{\langle a \rangle} \mathcal U$-module $M$. Finally, we set $D_{\langle a \rangle}^x M = \Phi_x D_{\langle a \rangle} M$ for any $\mathcal U$-module $M$ and call it the \emph{twisted localization} of $M$ relative to $a$ and $x$. We will use the notation $a^x\cdot m$  (or simply $a^x m$) for the element in  $D_{\langle a \rangle}^x M$ corresponding to $m \in D_{\langle a \rangle} M$. In particular, the following formula holds in $D_{\langle a \rangle}^{x} M$:
$$
u (a^{x} m) = a^{x} \left( \sum_{i\geq 0} \binom{-x}{i} \ad (a)^i (u) a^{-i}m\right)
$$
for $u \in \mathcal U$, $m \in D_{\langle a \rangle}  M$.

We will apply the twisted localization functor for  $({\mathcal U}, {\mathcal H})$ in the following two cases:

 \noindent (i) $\mathcal U = {\mathcal D}(n)^{\rm ev}$,   $\mathfrak h = \bigoplus_{i=1}^n \left( \mathbb C \medskip t_i\partial_{t_i} \right)$;

 \noindent (ii) $\mathcal U = U(\mathfrak{sp}(2n))$ with our fixed $\mathfrak h$.
 
In  case (i), for simplicity,  we will use the following notation:
$D_i = D_{\langle t_i^2 \rangle}$ and  $D_{\Gamma} = \prod_{i \in \Gamma} D_i$, where $\Gamma$ is a subset of $\llbracket n \rrbracket $. If  $\mat{\mu}=\sum_{k \in \Gamma} \mu_{k}  \varepsilon_k \in \mathfrak h^*$ (equivalently, $\mat{\mu} \in  \mathbb C^n)$, then we write $D_{\Gamma}^{\mu} = \prod_{k \in \Gamma} D^{\mu_k}_{k}$.

In case (ii), we will often consider the following setting. If $\Gamma$ is a subset of $\llbracket n \rrbracket $, and $\mat{\mu} \in  \mathbb C^n$, then $D_{\Gamma} = \prod_{k \in \Gamma} D_{\langle F_{k,-k}\rangle}$ and $D_{\Gamma}^{\mu} = \prod_{k \in \Gamma} D^{\mu_k}_{\langle F_{k,-k}\rangle}$. Note that, if $m \in M$ has weight $\lambda$, then the weight of $F_{k,-k}^xm$ is $\lambda + 2x\varepsilon_k$. Hence, for a weight module $M$,  $\Supp D_{\Gamma}^{\mat{\mu}} M \subset 2 \mat{\mu} + \Supp M + \mathcal Q_{C}$.

 \begin{example}
Let  $\overline{\mat{\nu}} = \mat{\nu} + \mathcal Q_C   \in \mathbb C^n/\mathcal Q_C$ be such that $ {\rm Int}(\mat{\nu}) \neq \llbracket n \rrbracket $. Then 
$$
\mathcal F(\overline{2\mat{\nu}}) \simeq D_{\llbracket n \rrbracket \setminus {\rm Int}(\nu)}^{\nu}  \mathcal F(\overline{0})
$$
as modules of $\mathfrak{sp}(2n)$ and of $\mathcal D(n)^{\rm ev}$.  In particular, $ D_1\mathcal F(\overline{\epsilon_1}) \simeq D_{1}^{1/2}  \mathcal F(\overline{0})$, or, equivalently, $ D_1\mathcal F^{\rm odd} \simeq D_{1}^{1/2}  \mathcal F^{\rm ev}$.

\end{example}
\begin{proposition} \label{sp-loc-inj} Let $\mat{\lambda}\in(\frac{1}{2}+\mathbb{Z})^{n}$ be a dominant $\mathfrak{so}(2n)$-weight.
\begin{itemize}
\item[(i)] Let $\Gamma$ be a subset of $\llbracket n \rrbracket $, and $\mat{\mu}=\sum_{k \in \Gamma} \mu_{k}  \varepsilon_k$ be such that  $2\mu_k \notin \mathbb Z$ for all $k$, in particular,  ${\rm Int} (2\mu) = \llbracket n \rrbracket \setminus \Gamma$. Then
$$ 
D_{\Gamma}^{\smat{\mu}} L_C(\mat{\lambda} + \mat{1}) \simeq   S(\chi_{\smat{\lambda} + \smat{1}\mat{}}, 2\mat{\mu}+ \mat{\lambda} + \mat{1}, \emptyset).
$$

\item[(ii)] Let $\overline{\mat{\nu}} = \mat{\nu} + \mathcal Q_C   \in \mathbb C^n/\mathcal Q_C$. The injective envelope of $S = S(\chi_{\smat{\lambda} + \smat{1}\mat{}}, \overline{\mat{\nu}}, \emptyset)$ in the category of bounded $\mathfrak{sp}(2n)$-modules is $D_{{\rm Hf} ( \mat{\overline{\nu}})} S  \simeq D_{\llbracket n \rrbracket }^{\nu} L_C(\lambda+1)$.
\end{itemize}
\end{proposition}
\begin{proof}
The  statements follow from \cite{Mat00} and \cite{GS06}. More precisely, for part (i) we use the fact that $D_{\Gamma}^{\mu} L_C(\lambda+1)$ and $S(\chi_{\smat{\lambda} + \smat{1}\mat{}}, 2\mat{\mu}+ \mat{\lambda} + \mat{1}, \emptyset)$ have the same support, central character and cone, while part (ii) follows from Proposition 4.7 in \cite{GS06}.
\end{proof}

\section{Tableaux realizations of bounded and generic $\mathfrak{sp}(2n)$-modules}
In this section we provide explicit tableaux realization of certain classes of infinite-dimensional $\mathfrak{sp}(2n)$-modules with explicit  action of the generators of $\mathfrak{sp}(2n)$ given by the formulas (\ref{Formulas sp}, \ref{Formulas sp1}, \ref{Formulas sp2}, \ref{Formulas sp3}). These classes include generic, simple bounded modules, and injective (hence projective) bounded modules.

\begin{definition}\label{good tableau}
A  tableau $T(L)$ of type $C$ is called \emph{$C$-regular} if its entries satisfy the following conditions: 
\begin{itemize}
\item[(i)] $\ell'_{ka}-\ell'_{ki}\neq 0$, $1\leq i\neq a\leq k\leq n$.
\item[(ii)] $\ell_{ka}-\ell_{ki}\neq 0$, $1\leq i\neq a\leq k\leq n-1$.
\item[(iii)] $\ell_{ka}+\ell_{ki}-1\neq 0$, $1\leq i\neq a\leq k\leq n-1$.
\item[(iv)] $2\ell_{ki}-1\neq 0$, $1\leq i\leq k\leq n-1$.
\end{itemize}
Equivalently, all rational functions in (\ref{operator A}, \ref{operator B}, \ref{operator C}, \ref{operator D}) are regular when evaluated at $T(L)$.

Similarly, a tableau $T(L)$ of type $D$ is called \emph{$D$-regular} if its entries satisfy conditions {\rm (i)--(iii)} above, cf. Remark \ref{rem-thm-d}. In particular, every $C$-regular tableau is $D$-regular. 
\end{definition}
In what follows we will show that the infinite-dimensional generic and bounded $\mathfrak{sp}(2n)$-modules have bases consisting of  regular tableaux. 
\subsection{Generic tableaux modules $V(T(L))$}

Consider  $\bar{X}:=\{x_{ij},x'_{ij}\ |\ 1\leq j\leq i\leq n\}$,  the field $\mathbb{C}(\bar{X})$ of rational functions on $\bar{X}$, and a tableau $T(X)$ of type $C$ with corresponding  entries $x_{ij},x'_{ij}$. By  $V(X)$ we denote the $\mathbb{C}(\bar{X})$-vector space with basis $$\mathcal{B}(X):=\{T(X+Z)\ |\ z_{ij},z'_{ij}\in \mathbb{Z}, \text{ and } z_{n1}=\ldots=z_{n,n}=0\},$$
where $z_{ij}, z'_{ij}$ are the corresponding entries of $T(Z)$.

For a tableau $T(W)$ of type $C$ denote by $W_{\rm top}$ the top row of $W$. For convenience, we will often consider  $W_{\rm top} \in \mathbb C^n$ as an $\mathfrak{sp}(2n)$-weight. 
\begin{theorem}\label{C-module variables}
The vector space $V(X)$ has a structure of $\mathfrak{sp}(2n)$-module with action of the generators of $\mathfrak{sp}(2n)$ given by the formulas (\ref{Formulas sp}, \ref{Formulas sp1}, \ref{Formulas sp2}, \ref{Formulas sp3}).  Moreover, the action of any $z\in Z(U(\mathfrak{sp}(2n)))$ on $V(X)$ is given by multiplication by 
$$
p_z(X_{\rm top} - \rho_C - \mat{\frac{1}{2}}) = p_{z}(x_{n1}+1/2,x_{n2}+3/2,\ldots,x_{nn}+(2n-1)/2).$$
\end{theorem}
\begin{proof}

The proof uses ideas of \cite{MZ00}, Section 6.4, and \cite{Zad17}, Section 3.5. To prove that $V(X)$ has the desired $\mathfrak{sp}(2n)$-module structure, we consider any element $u$ of the tensor algebra of $\mathfrak{sp}(2n)$, which is zero in $U(\mathfrak{sp}(2n))$, and show that $u\cdot T(Y)=0$ for any $T(Y)\in \mathcal{B}(X)$. After applying the formulas (\ref{Formulas sp}, \ref{Formulas sp1}, \ref{Formulas sp2}, \ref{Formulas sp3}) we present $u\cdot T(Y)$ as a linear combination of elements in the basis $\mathcal{B}(X)$ with the coefficients being rational functions in the variables $\{x_{ij},x'_{ij}\ |\ 1\leq j\leq i\leq n\}$, say $u\cdot T(Y)=\sum\limits_{p\in P}f_{p}(Y)T(Y+Z_p)$ for some integer tableaux $Z_{p}$. We will show that $f_{p}(Y)=0$ for any $p\in P$.

 Let us consider a fixed presentation of $u\in U(\mathfrak{sp}(2n))$. If $M$ is the maximum degree of a monomial in this presentation, we fix a number $N$ larger than $\max\{2|\lambda_1|, M^{2n}\}$. Let us consider any dominant $\mathfrak{sp}(2n)$ weight $\mat{\lambda}$, and $T(W)\in C_{st}^{\mat{\lambda}}$. Denote by $\mat{\bar{\lambda}}$ the dominant $\mathfrak{sp}(2n)$ weight $\mat{\lambda}-N(1,2,\ldots,n)$.  

Associated with $u$ we consider the tableau $T(S)$ of type $C$ with entries $s'_{ki}=-(2n-2k+i-1)N$, and $s_{ki}=-(2n-2k+i)N$. A straightforward verification shows that $T(\bar{W}):=T(W+S)\in C_{st}^{\bar{\mat{\lambda}}}$, and

\begin{itemize}
\item[(i)] $|\bar{w}'_{ka}-\bar{w}'_{ki}|> M^{2n}$ if $1\leq a\neq i\leq k\leq n$,
\item[(ii)] $|\bar{w}_{ka}-\bar{w}_{ki}|>M^{2n}$ if $1\leq a\neq i\leq k\leq n-1$,
\item[(iii)] $|\bar{w}_{ka}+\bar{w}_{ki}|>M^{2n}$ if $1\leq i\neq a\leq k\leq n-1$,
\item[(iv)] $|\bar{w}_{ki}-\frac{1}{2}|>M^{2n}$ for any $1\leq i\leq k\leq n-1$.
\end{itemize}

Therefore, when we write $u\cdot T(\bar{W})$ as linear combination of basis elements, the coefficients are precisely these rational functions that appear in $u\cdot T(Y)$ evaluated at the entries of $T(\bar{W})$, i.e.,  $u\cdot T(\bar{W})=\sum\limits_{p\in P}f_{p}(\bar{W})T(\bar{W}+Z_p)$. Since $T(\bar{W})\in C_{st}^{\bar{\mat{\lambda}}}$, we necessarily have $u\cdot T(\bar{W})=0$, and hence $f_{p}(\bar{W})=0$ for any $T(\bar{W})\in C_{st}^{\mat{\lambda}}$, and any dominant $\mathfrak{sp}(2n)$-weight $\mat{\lambda}$. If $\bar{f}_p (W) = {f}_p (W+S)$, then we have that 
$\bar{f}_{p}(W)=0$ for any $T(W)\in C_{st}^{\mat{\lambda}}$.
 Now, by Lemma \ref{density lemma} we have $\bar{f}_{p}=0$ for any $p\in P$, and hence ${f}_{p}=0$. Consequently,  $u\cdot T(Y)=0$.

For the second statement of the theorem we use similar reasoning as above.  Namely, set $z\in Z(U(\mathfrak{sp}(2n)))$, $T(Y)\in\mathcal{B}(X)$, and let   $z\cdot T(Y)=\sum_j g_j(Y)T(Y+Z_{j})$. If $Z_{j}\neq 0$, then $g_j(R)=0$ for any  tableau  $T(R)\in C_{st}^{\mat{\bar{\beta}}}$, and any dominant $\mathfrak{sp}(2n)$-weight $\mat{\beta}$. So, by Lemma \ref{density lemma} we conclude $z\cdot T(Y)= g(Y)T(Y)$ for some rational function $g \in \mathbb{C}(\bar{X})$. 

To compute explicitly $g(Y)$ we first consider $W$ such that $T(W) \in C_{st}^{\mat{\beta}}$ for a dominant $\mathfrak{sp}(2n)$-weight $\mat{\beta}$. Then $g(W) = p_z(\mat{\beta}) = p_z(W_{\rm top} - \rho_C - \mat{\frac{1}{2}})$. Consider $h(X):=g(X)-p_{z}(X_{\rm top}-\rho_C - \mat{\frac{1}{2}})$. Since $h(W) = 0$ for all $T(W) \in C_{st}^{\mat{\bar{\beta}}}$ for all dominant $\mathfrak{sp}(2n)$-weights $\mat{\beta}$, by Lemma \ref{density lemma} we have $h= 0$. But then we have that $z \cdot T(X) = p_{z}(X_{\rm top}-\rho_C - \mat{\frac{1}{2}}) T(X)$ as needed.

\end{proof}

\begin{definition}\label{def-generic}
A tableau $T(L)$ of type $C$ is called \emph{$C$-generic} if 

\begin{itemize}
\item[(i)] $\ell'_{ka}-\ell'_{ki}\notin\mathbb{Z}$, $1\leq i\neq a\leq k\leq n$.
\item[(ii)] $\ell_{ka}-\ell_{ki}\notin\mathbb{Z}$, $1\leq i\neq a\leq k\leq n-1$.
\item[(iii)] $\ell_{ka}+\ell_{ki}\notin\mathbb{Z}$, $1\leq i, a\leq k\leq n-1$.
\item[(iv)] $\ell_{ki}\notin\frac{1}{2}+\mathbb{Z}$ for any $1\leq i\leq k\leq n-1$.
\end{itemize}
A tableau $T(L)$ of type $D$ is called \emph{$D$-generic} if it satisfies conditions {\rm (i)--(iii)} above.
\end{definition}
Associated with  a generic tableau $T(L)$ of type $C$ (respectively, of type $D$), we have a discrete set of $C$-regular (respectively, $D$-regular) tableaux, $\mathcal{B}(T(L)):=\{T(L+Z)\ |\ z_{i,j},\ z'_{i,j}\in\mathbb{Z},\  \text{ and }z_{n,j}=0,  \text{ for any  }j\in \llbracket n \rrbracket \}.$ Set $V(T(L)):=\Span_{\mathbb{C}}\mathcal{B}(T(L))$.

\begin{corollary} \label{cor-sp-gen}
If $T(L)$ is a $C$-generic tableau, then $V(T(L))$ has a structure of an $\mathfrak{sp}(2n)$-module with action of the generators of $\mathfrak{sp}(2n)$ given by the formulas (\ref{Formulas sp}, \ref{Formulas sp1}, \ref{Formulas sp2}, \ref{Formulas sp3}). Moreover, the action of any $z\in Z(U(\mathfrak{sp}(2n)))$ is given by multiplication by $p_z(L_{\rm top} - \rho_C - \mat{\frac{1}{2}})$.
\end{corollary}
\begin{proof}

The statement follows  from Theorem \ref{C-module variables} after the evaluation $x_{ij}= \ell_{ij}$, and $x'_{ij}= \ell'_{ij}$. Indeed, the $C$-regularity of the set $\mathcal{B}(T(L))$ ensures that  any rational function that appears in the decomposition of $u\cdot T(Y)$ is defined for the evaluated values, and for any $u\in U(\mathfrak{sp}(2n))$.\end{proof}

Henceforth, we will call an $\mathfrak{sp}(2n)$-module \emph{generic} if it is isomorphic to 
a subquotient of $V(T(L))$ for some $C$-generic tableau $T(L)$.

In view of Remark \ref{rem-thm-d} we have the following corollary, the proof of which follows the same reasoning as the proof of Theorem \ref{C-module variables}, and consequently of Corollary \ref{cor-sp-gen}.

\begin{corollary} \label{cor-so-gen}
If $T(L)$ is a $D$-generic tableau, then $V(T(L))$ has a structure of an $\mathfrak{so}(2n)$-module with action of the generators of $\mathfrak{so}(2n)$ given by the GT-formulas of type $D$. Moreover, the action of any $z\in Z(U(\mathfrak{so}(2n)))$ is given by multiplication by $p_z(L_{\rm top} - \rho_D - \mat{\frac{1}{2}})$.
\end{corollary}

\begin{remark}
The connection between the $D$-type and the $C$-type generic modules $V(T(L))$ can be described as follows.
Take a tableau $T(L)$ which is both $C$-regular and $D$-generic, and let $\mat{\mu} \in \mathbb C^n$ be such that the tableau $T(\mat{\mu}, L)$ obtained by adding $\mat{\mu}$ as a column on the left of $T(L)$ is a $C$-generic tableau. Equivalently, $\mat{\mu}$ is such that
$\mu_k - \ell'_{ki} \notin \mathbb Z$ for all $k\geq i$. This leads to a correspondence $M\mapsto V(\mat{\mu}, M)$ from the set of modules that are both $C$-regular and $D$-generic  to the set of $C$-generic modules. In the next subsection we will observe similar correspondence $L_D(\mat{\lambda}) \mapsto V(\mat{\mu}, \mat{\lambda})$ from the $C$-regular $D$-standard modules (also known as spinor-type $\mathfrak{so}(2n)$-modules) to the infinite-dimensional bounded $\mathfrak{sp}(2n)$-modules (also known as oscillator-type $\mathfrak{sp}(2n)$-modules).
\end{remark}

\subsection{Bounded tableaux modules $V(\mat{\mu}, \mat{\lambda})$}

 Recall from Definition \ref{set of standard tableaux} that $D_{st}^{\mat{\lambda}}$ denotes the set of all $D$-standard tableaux with top row $\mat{\lambda}+ \rho_D + \mat{1/2}$. 
 \begin{definition}\label{def-T_D}
For any tableau $T(L)$ of type $C$, we denote by $T_{D}(L)$ the tableau of type $D$ obtained from $T(L)$ by removing the entries $\{\ell'_{11},\ell'_{21},\ldots,\ell'_{n1}\}$. Also, we denote by $T_{C\setminus D}(L)$ the $n$-tuple $(\ell'_{11},\ell'_{21},\ldots,\ell'_{n1})$.
\end{definition}

\begin{definition}\label{V space}
Given any   $\mat{\mu}=(\mu_{1},\ldots,\mu_{n})\in\mathbb{C}^{n}$ and $\mat{\lambda}=(\lambda_{1},\ldots,\lambda_{n})\in(\frac{1}{2}+\mathbb{Z}^{n})\cup \mathbb{Z}^{n}$ we associate the following:

\begin{itemize}
\item[(i)] $\mat{\ell}:=\mat{\lambda}+\rho_{D}+\mat{1/2}$, equivalently, $\ell_i := \lambda_i - i + 3/2$; 

\item[(ii)] the set of tableaux of type $C$
$$\mathcal{B}(\mat{\mu},\mat{\lambda}):=\{ T(W)\ |\ T_{C\setminus D}(W)\in\mat{\mu}+\mathbb{Z}^{n}\text{, and } T_{D}(W)\in D_{st}^{\mat{\lambda}}\};$$

\item[(iii)] the $\mathbb{C}$-vector space $V(\mat{\mu},\mat{\lambda})$ with basis $\mathcal{B}(\mat{\mu},\mat{\lambda})$.
\end{itemize}
\end{definition}

\begin{remark}\label{nonempty bases}
Since all tableaux in $\mathcal{B}(\mat{\mu},\mat{\lambda})$ have the same top row $\mat{\ell}$, it follows from the definition that 
$\mathcal{B}(\mat{\mu},\mat{\lambda})\neq\emptyset$ if and only if $\mat{\lambda}$ is a dominant $\mathfrak{so}(2n)$-weight.
\end{remark}

\begin{proposition}\label{Prop: good basis of tableaux}
Suppose that $\mat{\lambda}$ is a dominant $\mathfrak{so}(2n)$-weight. Then $\mathcal{B}(\mat{\mu},\mat{\lambda})$ is a $C$-regular set of tableaux if and only if $\mu_i \notin \mathbb Z$ and $\lambda_i \in \frac{1}{2} + \mathbb Z$ for all $i\in \llbracket n \rrbracket $.

\end{proposition}
\begin{proof}
If $\lambda_{i}\in\mathbb{Z}$ for some $i$, then $\mat{\lambda}\in\mathbb{Z}^n$, and $\mat{\ell}\in(\frac{1}{2}+\mathbb{Z})^n$. In particular, for any $T(R)\in \mathcal{B}(\mat{\mu},\mat{\lambda})$ we have $r_{11}\in\frac{1}{2}+\mathbb{Z}$. Moreover, as $T_{D}(R)$ is $D$-standard, we have $-r'_{22}\geq r_{11}>r'_{22}$, and $-r'_{22}\geq \frac{1}{2}>r'_{22}$. So, The tableau $T(S)=T(R+(\frac{1}{2}-r_{11})\delta^{(11)})\in \mathcal{B}(\mat{\mu},\mat{\lambda})$ and $s_{11}=\frac{1}{2}$.

Suppose $\mu_{k}\in\mathbb{Z}$ for some $k$, and $\lambda_{i}\notin \mathbb{Z}$ for any $i\in \llbracket n \rrbracket $. In this case we have $\ell_{k}\in \mathbb{Z}$, and given any tableau $T(R)\in \mathcal{B}(\mat{\mu},\mat{\lambda})$ we have $r'_{k,2}\in\mathbb{Z}$, and $T(S)=T(R+(r'_{k,2}-r'_{k,1})\delta'^{(k,1)})\in  \mathcal{B}(\mat{\mu},\mat{\lambda}) $ satisfies that $s'_{k1}=s'_{k2}$.
 To complete the proof we use that the entries of a $D$-standard tableaux satisfy the following properties:
\begin{itemize}
\item[(i)] $w_{ki}\neq w_{kj}$ for any $k\in \llbracket n \rrbracket $, and $1\leq i\neq j\leq k$,
\item[(ii)] $w'_{ki}\neq w'_{kj}$ for any $2\leq k\leq n$, and $2\leq i\neq j\leq k$, 
\item[(iii)]  $w_{k,i}+w_{k,j}\leq 0$   for any  $1\leq i\neq j\leq k$, and $2\leq k\leq n-1$.
\end{itemize}
The last property follows from $w_{k2}\in\{-\frac{1}{2},0\}+\mathbb{Z}_{\leq 0}$, $w_{k2}>\cdots>w_{kk}$, and $-w'_{k,2}\geq w_{k,1}>w'_{k2}\geq w_{ki}$ for any $i>1$.
\end{proof}

\emph{From now on, $\mat{\lambda}$ will denote any dominant $\mathfrak{so}(2n)$ weight in $(\frac{1}{2}+\mathbb{Z})^n$, $\mat{\mu}$ any vector in $(\mathbb{C}\setminus\mathbb{Z})^n$, and $\mat{\ell}:=\mat{\lambda}+\rho_{D}+\mat{\frac{1}{2}}=(\lambda_{1}-1+\frac{3}{2},\ldots,\lambda_{n}-n+\frac{3}{2})$.}

 Recall the definition of the polynomials $p_z$ from Remark \ref{center is polynomial in lambda}.

\begin{theorem}\label{module structure}
 The space $V(\mat{\mu}, \mat{\lambda})$ has a structure of a $\mathfrak{sp}(2n)$-module, with action of the generators given by the formulas (\ref{Formulas sp}, \ref{Formulas sp1}, \ref{Formulas sp2}, \ref{Formulas sp3}, \ref{operator A}, \ref{operator B}, \ref{operator C}, \ref{operator D}). Moreover, the action of any $z\in Z(U(\mathfrak{sp}(2n)))$ is given by multiplication by $p_{z}(\mat{\lambda}+\mat{1})$, i.e. $V(\mat{\mu}, \mat{\lambda})$ has central character $\chi_{\smat{\lambda} + \smat{1}\mat{}}$.
\end{theorem}

\begin{proof}

We use the same reasoning as in the proof of Theorem \ref{C-module variables}. For reader's convenience, we outline the main idea again.
By Proposition \ref{Prop: good basis of tableaux}, the basis $\mathcal{B}(\mat{\mu}, \mat{\lambda})$ of $V(\mat{\mu}, \mat{\lambda})$ consists of $C$-regular tableaux. In order to prove that $V(\mat{\mu}, \mat{\lambda})$ has a structure of $\mathfrak{sp}(2n)$-module, it is sufficient to prove that $u\cdot T(W)=0$ for any any $T(W)\in \mathcal{B}(\mat{\mu}, \mat{\lambda})$ and any $u$ in the tensor algebra of $\mathfrak{sp}(2n)$ which is zero in $U(\mathfrak{sp}(2n))$. Let us fix such $u$ and $T(W)$.

We can write  $u\cdot T(W)=\sum_if_i(W)T(W+Z_i)$, where $f_i$ are the rational functions that appear applying the formulas (\ref{operator A}, \ref{operator B}, \ref{operator C}, \ref{operator D}), $Z_i$ are tableaux with integer entries, and $T(W+Z_i)$ are tableaux in $\mathcal{B}(\mat{\mu}, \mat{\lambda})$.

Let $M$ be the maximum degree of the monomials $u_j$ such that $u = \sum_j u_j$.
Fix an even integer $N$ bigger than $\max\{2|\lambda_{1}|,\ M^{2n}\}$.
  Note that for any dominant $\mathfrak{sp}(2n)$-weight $\mat{\beta}$, the weight $\mat{\bar{\beta}}:=\mat{\beta}+\mat{\lambda-\mat{N}-\mat{3/2}}$ is also  a dominant $\mathfrak{sp}(2n)$-weight. 
  
  Let $T(S)$ be the tableau of type $C$ with entries: $s_{ki}=w_{ki}-N-3/2$, $s'_{ki}=w'_{ki}-N-3/2$ for $i\neq 1$, and $$s'_{k1}=\begin{cases}

\lambda_1- N-3/2, & \text{  if  } \mu_{k}\in\frac{1}{2}+\mathbb{Z},\\

\lambda_1-\frac{N}{2}-3/2, & \text{  if  } \mu_{k}\notin\frac{1}{2}+\mathbb{Z}.
\end{cases}$$

Note that for any $T(R)\in C_{st}^{\beta}$ we have  $T(\bar{R}):=T(R+S)\in C_{st}^{\mat{\bar{\beta}}}$, and for any integer tableau $Z$ we have  $T(\bar{R}+Z)\in C_{st}^{\mat{\bar{\beta}}}$ if and only if $T(W+Z)\in \mathcal{B}(\mat{\mu}, \mat{\lambda})$.   Therefore $u\cdot T(\bar{R})=\sum_if_i(\bar{R})T(\bar{R}+Z_i)$. But since $u\cdot T(\bar{R})=0$ we obtain that  $f_i(\bar{R})=0$ for any $T(\bar{R})\in C_{st}^{\bar{\beta}}$ and any dominant $\mathfrak{sp}(2n)$-weight $\beta$.
 We now apply Lemma \ref{density lemma} to the polynomial $\bar{f_i}(X) = f_i(X+S)$ and
 conclude $\bar{f}_i=0$. Hence $f_i =0$ and consequently  $u\cdot T(W)=0$ as desired.

For the second statement, consider $z\in Z(U(\mathfrak{sp}(2n)))$ and $T(W)\in\mathcal{B}(\mat{\mu},\mat{\lambda})$, and let   $z\cdot T(W)=\sum_j g_j(W)T(W+Z_{j})$. If $Z_{j}\neq 0$, then $g_j(R)=0$ for any  tableau  $T(R)\in C_{st}^{\mat{\bar{\beta}}}$, and any dominant $\mathfrak{sp}(2n)$ weight $\mat{\beta}$. So, by Lemma \ref{density lemma} we conclude $z\cdot T(W)= g(W)T(W)$ for some rational function $g$.  To compute explicitly $g(W)$ we first consider $W$ such that $T(W) \in C_{st}^{\mat{\beta}}$ for a dominant $\mathfrak{sp}(2n)$ weight $\mat{\beta}$. Then $g(W) = p_z(\mat{\beta}) = p_z(W_{\rm top} - \rho_C - \mat{\frac{1}{2}})$.  Consider $h(X):=g(X)-p_{z}(X_{\rm top}-\rho_C - \mat{\frac{1}{2}})$. Since $h(W) = 0$ for all $T(W) \in C_{st}^{\mat{\bar{\beta}}}$ for all dominant $\mathfrak{sp}(2n)$-weights $\mat{\beta}$, by Lemma \ref{density lemma} we have $h= 0$. 
But then for $T(U) \in\mathcal{B}(\mat{\mu},\mat{\lambda})$, we have that
$$
z \cdot T(U) = p_{z}(U_{\rm top}-\rho_C - \mat{\frac{1}{2}}) T(U) = p_{z}(\mat{\lambda} + \rho_D + \mat{\frac{1}{2}} -\rho_C - \mat{\frac{1}{2}}) T(U) = p_{z}(\mat{\lambda} + \mat{1}) T(U).
$$
\end{proof}

\subsection{Support, primitive vectors, and weight multiplicities of $V(\mat{\mu}, \mat{\lambda})$}

To compute the weights of the basis vectors of $V(\mat{\mu}, \mat{\lambda})$ we need the following definition.
\begin{definition}

Given any tableau $T(L)$ of type $C$, we define the \emph{$C$-weight of $T(L)$} as $\omega(T(L))=(\omega_1(L),\omega_2(L),\ldots,\omega_n(L))$, where 
\begin{align}\label{C-weights}
\omega_k(L):=2\sum_{i=1}^{k}\ell'_{ki}-\sum_{i=1}^{k}\ell_{ki}-\sum_{i=1}^{k-1}\ell_{k-1,i}+k-\frac{1}{2}.
\end{align}

In particular, the $C$-weight of $T(L)$ in $\mathcal B (\mat{\mu}, \mat{\lambda})$ is its $\mathfrak 
 h$-weight written as an element of $\mathbb C^n$.   
\end{definition}
\begin{remark} \label{rem-sum}One way to write the formula for $\omega(T(L))$ in a more compact form is to use notation for the sum of the entries if the $k$th and $k'$th row of $T(\ell)$: $S_k(\ell') =\sum_{i=1}^{k}\ell'_{ki}$ and $S_k(\ell) =\sum_{i=1}^{k}\ell_{ki}$. Then 
$$
\omega(T(L)) = 2 S(\ell') - S(\ell) - S(\ell_{-1}) - \rho_C - \mat{\frac{1}{2}},
$$
where $S(\ell) = (S_1(\ell),S_2(\ell),...,S_n(\ell))$, $S(\ell_{-1}) = (0,S_1(\ell),...,S_{n-1}(\ell))$. 
\end{remark}

We next introduce two important tableaux in $\mathcal B(\mat{\mu}, \mat{\lambda})$.

\begin{definition}\label{T(W)}
Let $\mat{\mu}=(\mu_{1},\ldots,\mu_{n})\in\mathbb{C}^{n}$, $\mat{\lambda}\in(\frac{1}{2}+\mathbb{Z})^{n}$ be a dominant  $\mathfrak{so}(2n)$-weight. Recall that $\mat{\ell}=\mat{\lambda}+\rho_{D}+\mat{\frac{1}{2}}\in\mathbb{Z}^{n}$. 

\begin{itemize}
\item[(i)] By $T(W^{ \smat{\mu},\smat{\lambda}\mat{}})$ we denote the tableau of type $C$ with entries $w'_{k1}=\mu_{k}$, $w'_{ki}=\ell_{i}$ for  $i\geq 2$, and $w_{ki}=\ell_{i}$ for $i\geq 1$. Namely,
$$
\begin{small}
T(W^{ \smat{\mu},\smat{\lambda}\mat{}})=\begin{tabular}{cccccccccc}
\xymatrixrowsep{0.2cm}
\xymatrixcolsep{0.1cm}\xymatrix @C=0.5em {
&\ell_{1}& &\ell_{2}& &\ell_{3}&&\cdots&&\ell_{n}\\
\mu_{n}& &\ell_{2}& &\ell _{3}&&\cdots&&\ell_{n}&\\
&\ell_{1}& &\ell_{2}& &\cdots&&\ell_{n-1}&&\\
\mu_{n-1}& &\ell_{2}& &\cdots&&\ell_{n-1}&&&\\
&\vdots& &\vdots&&\vdots&&&&\\
\vdots& &\vdots&&\vdots&&&&&\\
&\ell_{1}& &\ell_{2}& &&&&&\\
\mu_{2}& &\ell_{2}& &&&&&&\\
&\ell_{1}& && &&&&&\\
\mu_{1}& && &&&&&&\\
}\end{tabular}
\end{small}
$$

\item[(ii)] By  $T(W_{ \smat{\mu},\smat{\lambda}\mat{}})$ we denote the tableau of type $C$ with entries $w'_{k1}=\mu_k$, $w_{n1}=\ell_1$, $w_{k-1,1}=1-w_{k1}$ for $2\leq k\leq n$, and 
 $w_{k, i}=w'_{ki}=\ell_i$, for any $i \geq 2$. The tableaux $T(W_{ \smat{\mu},\smat{\lambda}\mat{}})$ for $n=3$ and $n=4$, respectively, are as follows
{\small $$\ \ \ \ \ \ \ \ \begin{tabular}{cccccccccc}
\xymatrixrowsep{0.2cm}
\xymatrixcolsep{0.1cm}\xymatrix @C=0.5em {
&\ell_1& &\ell_2&&\ell_3&&&&\\
\mu_{3}& &\ell_2&&\ell_3&&&&&\\
&1-\ell_{1}& &\ell_{2}& &&&&&\\
\mu_{2}& &\ell_{2}& &&&&&&\\
&\ell_{1}& && &&&&&\\
\mu_{1}& && &&&&&&\\
}\end{tabular}\begin{tabular}{cccccccccc}
\xymatrixrowsep{0.2cm}
\xymatrixcolsep{0.1cm}\xymatrix @C=0.5em {
&\ell_{1}& &\ell_{2}& &\ell_{3}&&\ell_4&\\
\mu_{4}& &\ell_{2}& &\ell _{3}&&\ell_4&&\\
&1-\ell_{1}& &\ell_{2}& &\ell_3&&\\
\mu_{3}& &\ell_{2}& &\ell_3&&&\\
&\ell_1& &\ell_2&&&&\\
\mu_{2}& &\ell_2&&&&&\\
&1-\ell_{1}& && &&&\\
\mu_{1}& && &&&&\\
}\end{tabular}$$}
The values  $\mu_i= \mat{\frac{1}{2}}$ will play a special role in the paper as in this case the tableaux are highest weight vectors of submodules of $V(\mat{\mu}, \mat{\lambda})$. We set  $T(W_{\smat{\lambda}\mat{}}) := T(W_{ \smat{\frac{1}{2}},\smat{\lambda}\mat{}})$.
\end{itemize}
 \end{definition}

Recall that $\tau_{1}\mat{\lambda}=(-\lambda_1,\lambda_2,\ldots,\lambda_n)$. 

\begin{lemma}\label{T(W)-weight}
With the notation of  Definition \ref{T(W)} we have the following.
\begin{itemize}
\item[(i)]
$T(W^{ \smat{\mu},\smat{\lambda}\mat{}})\in \mathcal{B}(\mat{\mu},\mat{\lambda})$ and $\omega(T(W^{ \smat{\mu},\smat{\lambda}\mat{}}))=2\mat{\mu}+\mat{\lambda}-(2\lambda_1)\mat{1}$.

\item[(ii)]  $T(W_{ \smat{\mu},\smat{\lambda}\mat{}})\in \mathcal{B}(\mat{\mu},\mat{\lambda})$ and  
$$\omega(T(W_{ \smat{\mu},\smat{\lambda}\mat{}}))=\begin{cases}
2\mat{\mu}+\mat{\lambda}, & \text{ if $n$ is even,}\\
2\mat{\mu}+\tau_{1}\mat{\lambda}, & \text{ if $n$ is odd.}
\end{cases}$$
\end{itemize}
\end{lemma}

\begin{proof}
All statements follow by a direct verification. For reader's convenience we prove the formula for $\omega(T(W^{ \smat{\mu},\smat{\lambda}\mat{}}))$. Using Formula (\ref{C-weights}), the $k$-th entry of $\omega(T(W^{ \smat{\mu},\smat{\lambda}\mat{}}))$ equals
\begin{align*}
2\mu_{k}+2\sum_{i=2}^{k}\ell_{i}-\sum_{i=1}^{k}\ell_{i}-\sum_{i=1}^{k-1}\ell_{i}+k-\frac{1}{2}=& 2\mu_{k}-2\ell_{1}+\ell_k+k-\frac{1}{2}\\
=& 2\mu_{k}+\lambda_k+1+(-2\lambda_1-1),
\end{align*}
as needed. \end{proof}

Let $\mathfrak n$ be the nilpotent radical of the fixed Borel subalgebra of $\mathfrak{sp}(2n)$, in particular, $\Delta_{\mathfrak n} = \{ \varepsilon_i \pm \varepsilon_j, -\varepsilon_k\; | \; i < j  \}$.

\begin{theorem} \label{HW tableau}
The tableau $T(W_{\smat{\lambda}\mat{}})$ is an $\mathfrak n$-primitive vector of $V(\mat{\frac{1}{2}},\mat{\lambda})$.

\end{theorem}

\begin{proof}  To prove that $\mathfrak n \cdot T(W_{\smat{\lambda}\mat{}}) =0$, we show that $F_{k-1,k}T(W_{\smat{\lambda}\mat{}})=0$ for any $k=2,\ldots,n$, and $F_{-j,j}T(W_{\smat{\lambda}\mat{}}) =0$ for any $j\in \llbracket n \rrbracket $. The proof of these identities is a bit technical and is included in the Appendix. 

 \begin{lemma}\label{family of D-standard}
 For $k<n$,  
 $$f_k(T(R)):=\begin{cases}
T(R+\delta^{(k1)}), & \text{ if  \ }r_{k1}\neq -r'_{k2} \text{ and } r_{k1}\neq -r'_{k+1,2},\\
T(R-\delta^{(k1)}), &\text{ otherwise}\\
\end{cases}$$
 defines a function $f_k:D_{st}^{\mat{\lambda}}\to D_{st}^{\mat{\lambda}}$.
 \end{lemma}

 \begin{proof}
We need to show that $f_k(T(R))$ is $D$-standard with top row $\mat{\ell}$. Since $T(R)$ is $D$-standard, and the only entry of $T(S):=f_k(T(R))$ that is different from the correspondent entry of $T(R)$ is the one in position $(k,1)$, it is sufficient to prove that $-r'_{k,2}\geq s_{k,1}>r'_{s2}$ and $-r'_{k+1,2}\geq s_{k,1}>r'_{k+1,2}$. We have two cases:
\begin{itemize}
\item[Case 1:] $s_{k,1}=r_{k1}+1$. In this case $r_{k1}\neq -r'_{k2}$ and $r_{k1}\neq -r'_{k+1,2}$. Hence\\ $-r'_{k,2}\geq r_{k,1}+1>r'_{k2}$ and $-r'_{k+1,2}\geq r_{k,1}+1>r'_{k+1,2}$, as needed.
\item[Case 2:] $s_{k,1}=r_{k1}-1$. In this case, since $r_{ij}\in\mathbb{Z}$ for any $i,j$, we necessarily have $r_{k,1}\in\mathbb{Z}_{>0}$, and $-r'_{k,2}\geq r_{k,1}-1\geq 0>r'_{k2}$. Also, $-r'_{k+1,2}\geq r_{k,1}-1\geq 0>r'_{k+1,2}$ implies $-r'_{k,2}\geq r_{k,1}+1>r'_{s2}$ and $-r'_{k+1,2}\geq r_{k,1}+1>r'_{k+1,2}$.
\end{itemize}
 \end{proof}

\begin{corollary}\label{many D-standard}
For any $T(R)$ in $D_{st}^{\mat{\lambda}}$ and $A:=\{a_1,\ldots,a_s\}\subseteq \llbracket n-1 \rrbracket $ the tableau $f_{A}(T(R)):=f_{a_1}\circ f_{a_2}\circ \cdots\circ f_{a_{s}}(T(R))$ belongs to $D_{st}^{\mat{\lambda}}$. Moreover, $f_{A}(T(R))=f_{B}(T(R))$ if and only if $A=B$.
\end{corollary}

\begin{remark}\label{same weight}
Note that if two tableaux $T(R), T(R+Z)\in \mathcal{B}(\mat{\mu},\mat{\lambda})$ have the same $C$-weight, then $\sum\limits_{i=1}^{k}z_{ki}\in 2\mathbb{Z}$ for any $k\in \llbracket n-1 \rrbracket $.
\end{remark}

Recall that  $\mathbf{\mathcal{Q}}_C=\mathbb Z \Delta=\{(a_1,\ldots, a_n)\in\mathbb{Z}^{n}\ |\ a_1+\cdots+a_n\in2\mathbb{Z}\}$.

\begin{proposition}\label{support}
 $\Supp(V(\mat{\mu},\mat{\lambda}))=2\mat{\mu}+\mat{\lambda}+\mat{1}+\mathbf{\mathcal{Q}}_C$. 
\end{proposition}
\begin{proof}
Let  $T(W^{ \smat{\mu},\smat{\lambda}\mat{}})$ be as in Definition \ref{T(W)}. Note that $T(W^{( \smat{\mu},\smat{\lambda}\mat{})}
+m\delta'^{(k1)})\in \mathcal{B}(\mat{\mu},\mat{\lambda})$ for any $m\in \mathbb{Z}$ and any $k\in \llbracket n \rrbracket $. Moreover, using Lemma \ref{T(W)-weight} we have 
$$\omega\left(\left\{T\left(W^{( \smat{\mu},\smat{\lambda}\mat{})}+\sum_{k=1}^{n}m_k\delta'^{(k1)}\right)\ |\ (m_{1},\ldots,m_n)\in\mathbb{Z}^n\right\}\right) =2\mat{\mu}+\mat{\lambda}+\mat{1}+(2\mathbb{Z})^n.$$
Furthermore, by Lemma \ref{family of D-standard}, for any $T(R)\in \mathcal{B}(\mat{\mu},\mat{\lambda})$ and $i\in \llbracket n-1 \rrbracket $, we have $T(R+\delta^{(i1)} )\in \mathcal{B}(\mat{\mu},\mat{\lambda})$ or $T(R-\delta^{(i1)} )\in \mathcal{B}(\mat{\mu},\mat{\lambda})$. On the other hand, $\omega(T(R\pm\delta^{(i1)}))=\omega(T(R))\mp(\varepsilon_{i}+\varepsilon_{i+1})$. This completes the proof. \end{proof}

\begin{theorem}\label{number of D-standard tableaux}
 For any  $\mat{\gamma}\in\Supp(V(\mat{\mu},\mat{\lambda}))$ we have 
\begin{equation}\label{dimension of weight spaces}
\dim V(\mat{\mu},\mat{\lambda})^{\gamma}=\frac{1}{2^{n-1}}\dim L_{D}(\mat{\lambda}).
\end{equation}
\end{theorem}
\begin{proof} 
Recall the notation $S_k(L)$ from Remark \ref{rem-sum}.
 Consider the equivalence relation on $D_{st}^{\mat{\lambda}}$ defined by $T(R)\sim T(M)$ if and only if  $S_{k}(R-M)\in 2\mathbb{Z}$ for any $k\in \llbracket n-1 \rrbracket $. Note that for any $T(S)\in D_{st}^{\mat{\lambda}}$ the set $\{f_{A}(T(S))\ |\ A\subseteq \llbracket n-1 \rrbracket \}$, where $f_{A}$ are the functions constructed in Corollary \ref{many D-standard}, is a complete set of representatives for $D_{st}^{\mat{\lambda}}/\sim$. Hence  we have $\frac{1}{2^{n-1}}\dim(L_{D}(\mat{\lambda}))$ equivalence classes.

Given $T(R)\in\mathcal{B}(\mat{\mu},\mat{\lambda})$, we will show that $\{T(W)\in \mathcal{B}(\mat{\mu},\mat{\lambda})\ |\  \omega(T(W))=\omega(T(R))\}$ is in a one-to-one correspondence with the elements of the equivalence class $[T_{D}(R)]$ of $T_{D}(R)$.

Indeed, for any $T(M)\in D_{st}^{\mat{\lambda}}$ such that $T(M)\sim T_{D}(R)$ we consider the tableau $T(\widetilde{M})$ of type $C$ such that $T_{D}(\widetilde{M})=T(M)$, and $T_{C\setminus D}(\widetilde{M})=(\tilde{m}'_{11},\ldots,\tilde{m}'_{n1})$ with 
$\tilde{m}'_{11}=\frac{1}{2}S_{1}(M)$, and $\tilde{m}'_{k1}=r'_{ki}-S_{k}'(M)+\frac{1}{2}S_{k}(M)+\frac{1}{2}S_{k-1}(M)$ for $k\geq 2$. A straightforward computation shows that $\omega(T(\widetilde{M}))=\omega(T(R))$, as desired. For the reverse direction we use Remark \ref{same weight}.\end{proof}

\begin{remark}\label{explicit weight spaces}
The proof of Theorem \ref{number of D-standard tableaux} provides also an explicit basis of the weight spaces $V(\mat{\mu}, \mat{\lambda})^{\gamma}$. Given $T(W)\in \mathcal{B}(\mat{\mu},\mat{\lambda})$, every tableau $T(R)\in[T_{D}(W)]$ can be extended in a unique way to a tableau $T(\widetilde{R})$ in $\mathcal{B}(\mat{\mu},\mat{\lambda})$ such that  $\omega(T(\widetilde{R}))=\omega(T(W))$. Indeed, consider $T_D(\widetilde{R})=T(R)$, and 
$$\tilde{r}'_{k1}:=w'_{k,1}+S_k'(T_{D}(W)-T(R))-\frac{1}{2}S_{k}(T_{D}(W)-T(R))-\frac{1}{2}S_{k-1}(T_{D}(W)-T(R)).$$
\end{remark}

\begin{corollary}
The module $V(\mat{\mu}, \mat{\lambda})$ is a bounded $\mathfrak{sp}(2n)$-module of degree $d(\mat{\lambda}):=\frac{1}{2^{n-1}}\dim(L_{D}(\mat{\lambda}))$.

\end{corollary}

\begin{lemma}
For any dominant $\mathfrak{so}(2n)$-weight $\mat{\lambda}\in(\frac{1}{2}+\mathbb{Z})^n$ the equation  $$
2\mat{\mat{x}}+\mat{\lambda}+\mat{1}\equiv\mat{\lambda}+\varepsilon_{1}\mod \mathcal{Q}_C
$$
has no solution $\mat{x}$ in $(\mathbb{C}\setminus\mathbb{Z})^n$.

\end{lemma}

\begin{proof}
If $\mat{x}$ is a solution of the given equation and $\mat{x}\in (\mathbb{C}\setminus\mathbb{Z})^n$, we have $2\mat{x}+\mat{1}-\varepsilon_{1}\in \mathcal{Q}_C$. This implies
$x_{i}\in\frac{1}{2}+b_i$ with $b_i\in\mathbb{Z}$ for all $i\in \llbracket n \rrbracket $.
However, $n-1+2\sum\limits_{i=1}^n x_{i}=(n-1)+n+2\sum\limits_{i=1}^n b_{i}$ is an odd integer, which is a contradiction. \end{proof}

\begin{proposition} \label{prop-missing}
For any dominant $\mathfrak{so}(2n)$-weight $\mat{\lambda}\in(\frac{1}{2}+\mathbb{Z})^n$ consider the map $\Psi_{\smat{\lambda}\mat{}}:\left(\mathbb{C}\setminus\mathbb{Z}\right)^{n}\longrightarrow \mathbb C^n/\mathcal{Q}_C$ defined by $\Psi_{\smat{\lambda}\mat{}}(\mat{\mu})=\Supp(V(\mat{\mu}, \mat{\lambda}))+\mathcal{Q}_C$. Then $\im(\Psi_{\smat{\lambda}}\mat{})$ is the complement of $\mat{\lambda}+\varepsilon_1+\mathcal{Q}_C$ in $\mathbb C^n/\mathcal{Q}_C$.
\end{proposition}
\begin{proof}
We first apply the preceding lemma to prove that  $\Psi_{\smat{\lambda}\mat{}}(\mat{\mu})$ is contained in the complement of $\mat{\lambda}+\varepsilon_1+\mathcal{Q}_C$. Then we proceed by induction on $n$. For $n=1$ the statement is obvious. Consider now $n>1$ and let $\mat{a}\in\mathbb{C}^n$ be such that $\mat{a}\notin \mat{\lambda}+\varepsilon_1+\mathcal{Q}_C$. If $a_1 - \lambda_1 -1 \notin 2 \mathbb Z$, we choose $\mu_1 = \frac{a-\lambda_1-1}{2}$ and set $\mat{a'} = (a_2,...,a_n)$, $\mat{\lambda'} = (\lambda_2,...,\lambda_n)$. If $a_1 - \lambda_1 -1 \in 2 \mathbb Z$,  we choose $\mu_1 = \frac{a-\lambda_1}{2}$ and set $\mat{a'} = (a_2,...,a_n) + \varepsilon_2$, $\mat{\lambda'} = (\lambda_2,...,\lambda_n)$. In both cases $\mat{a'}\notin \mat{\lambda'}+\varepsilon_2+\mathcal{Q}_C$ and we may apply the induction hypothesis to find $\mat{\mu'} = (\mu_2,...,\mu_n) \in \left(\mathbb{C}\setminus\mathbb{Z}\right)^{n-1}$ such that $2\mat{\mat{\mu'}}+\mat{\lambda'}+\mat{1}=\mat{a'}\mod \mathcal{Q}_C$.
\end{proof}

\subsection{Subquotients of  $V(\mat{\mu}, \mat{\lambda})$}
Let $k\in \llbracket n \rrbracket $ and 
consider $\mat{\mu}\in\left(\mathbb{C}\setminus\mathbb{Z}\right)^{n}$ such that $\mu_{k}\in\frac{1}{2}+\mathbb{Z}$. Then set:
$$
 \mathcal{B}_k^{+}(\mat{\mu},\mat{\lambda}):=\left\{T(W)\in \mathcal{B}(\mat{\mu},\mat{\lambda})\ |\  w'_{k,1}-\frac{1}{2}\in\mathbb{Z}_{\geq 0}\right\}; \; V_k^{+}(\mat{\mu}, \mat{\lambda}):=\Span_\mathbb{C} \mathcal{B}_k^{+}(\mat{\mu},\mat{\lambda}).$$ Recall that $\mathcal C_k = \left\{ x \in \mathcal Q_C \; | \; \langle x,\varepsilon_k \rangle \geq 0 \right\}$.

\begin{theorem}\label{submodule structure}
 If $\mu_{k}\in\frac{1}{2}+\mathbb{Z}$, then $V_k^{+}(\mat{\mu}, \mat{\lambda})$ is a submodule of $V(\mat{\mu}, \mat{\lambda})$ such that  $\Supp V_k^{+}(\mat{\mu}, \mat{\lambda}) \subset  \mat{\eta}+ \mathcal C_k$ for some $\mat{\eta}\in\Supp V_k^{+}(\mat{\mu}, \mat{\lambda}) $.
\end{theorem}
\begin{proof}
To show that $V_k^{+}(\mat{\mu}, \mat{\lambda})$ is a submodule we look at the actions of the generators of $\mathfrak{sp}(2n)$ on a tableau  $T(L)  \in \mathcal{B}_k^{+}(\mat{\mu},\mat{\lambda})$ and see that the tableau $T(L+S-\delta^{(k,1)'})$ that appear in the decomposition is either in $\mathcal{B}_k^{+}(\mat{\mu},\mat{\lambda})$ or the corresponding coefficient $B_{k1}(L+S)$ is zero. For the support, we consider $A=\{T(R)\in \mathcal{B}(\mat{\mu},\mat{\lambda})\ |\ T_{C\setminus D}(R)=\mat{\mu}+(1/2-\mu_k)\varepsilon_{k}\}$. The set $A$ has $\dim L_{D}(\mat{\lambda})$ elements, so we consider $T(S)\in A$ such that $\omega_{k}(T(S))$ is minimum. Then $T(S)\in \mathcal{B}_k^{+}(\mat{\mu},\mat{\lambda})$, and $\mat{\eta}=\omega_{C}(T(S))$ satisfies the condition of the theorem. \end{proof}

\begin{definition}\label{max}
For any dominant $\mathfrak{so}(2n)$ weight $\mat{\lambda}$, and $k\in \llbracket n-1 \rrbracket $ we consider 
$$t_{k}(\mat{\lambda}):=\max_{T(R),T(W)\in D_{st}^{\smat{\lambda}\mat{}}}\left\{\left\lfloor S_k'(R-W)-\frac{1}{2}S_{k+1}(R-W)-\frac{1}{2}S_k(R-W)\right\rfloor\right\}.$$

\end{definition}

\begin{lemma}\label{essential support} 
Set $\mat{\mu}\in\left(\mathbb{C}\setminus\mathbb{Z}\right)^{n}$, and $\mat{\lambda}\in(\frac{1}{2}+\mathbb{Z})^{n}$ a dominant $\mathfrak{so}(2n)$-weight. If $\mu_{k}\in\frac{1}{2}+\mathbb{Z}$, then for any tableau $T(W)\in \mathcal{B}_k^{+}(\mat{\mu},\mat{\lambda})$ such that $w'_{k1}-1/2\geq t_k(\mat{\lambda})$ we have $\omega(T(W))\in \Supp_{\rm ess} V(\mat{\mu}, \mat{\lambda})$. In particular, 
$$
\deg V_k^{+}(\mat{\mu}, \mat{\lambda}) = \deg V(\mat{\mu}, \mat{\lambda}) = \frac{1}{2^{n-1}} \dim L_D (\mat{\lambda}).$$
\end{lemma}
\begin{proof}
By Remark \ref{explicit weight spaces} it is enough to prove that $T(\widetilde{R})\in \mathcal{B}_k^{+}(\mat{\mu},\mat{\lambda})$ for any $T(R)\in [T_{D}(W)]$. However, by definition 
{\small
\begin{align*}
\tilde{r}'_{k1}&:=w'_{k1}+S_k'(T_{D}(W)-T(R))-\frac{1}{2}S_{k+1}(T_{D}(W)-T(R))-\frac{1}{2}S_k(T_{D}(W)-T(R))\\
&\geq t_{k} (\mat{\lambda})+1/2+S_k'(T_{D}(W)-T(R))-\frac{1}{2}S_{k+1}(T_{D}(W)-T(R))-\frac{1}{2}S_k(T_{D}(W)-T(R))\\
&\geq 0.
\end{align*}
}
\end{proof}

Lemma \ref{essential support}, and Theorem \ref{submodule structure} imply that
$$
\mat{\lambda}' + \mathcal C_k \subset \mbox{Supp}_{\rm ess} V_k^{+}(\mat{\mu}, \mat{\lambda})\subset  \mbox{Supp} V_k^{+}(\mat{\mu}, \mat{\lambda}) \subset \mat{\lambda}'' + \mathcal C_k
$$
for some weights $\mat{\lambda}' , \mat{\lambda}''$ in  $\mbox{Supp} V_k^{+}(\mat{\mu}, \mat{\lambda}) $. 
In particular, $\mathcal C (V_k^{+}(\mat{\mu}, \mat{\lambda})) = \mathcal C_k$. 

 The following is straightforward.

\begin{lemma}\label{properties of cones} Let $M$ and $N$ be nontrivial proper submodules of $V(\mat{\mu},\mat{\lambda})$ that have cones $\mathcal C(M), \mathcal C(N)$, respectively. Then the following identities hold
\begin{itemize}
\item[(i)] $M+N$ has cone $\mathcal C(M+N) = \mathcal C(M) \cup \mathcal C(N)$.
\item[(ii)] If $N \subset M$ and $N\neq M$, then $M/N$ has cone $\mathcal C(M/N)=\mathcal C(M)\setminus \mathcal C(N)$.
\item[(iii)] If $M\cap N \neq 0$, then $M \cap N$ has cone $\mathcal C(M \cap N) = \mathcal C(M) \cap \mathcal C(N)$.
\item[(iv)] $\mathcal C (V_k^{+}(\mat{\mu}, \mat{\lambda})) = \mathcal C_k$ and $\mathcal C (V (\mat{\mu}, \mat{\lambda})/V_k^{+}(\mat{\mu}, \mat{\lambda})) = -\mathcal C_k$.
\end{itemize}
\end{lemma}

\begin{definition}\label{def-Sigma} Let $\Sigma$ be a proper nonempty subset  of $ {\rm Int} (2{\mat{\mu}} + \mat{\lambda} + \mat{1/2}) = {\rm Int} (2{\mat{\mu}}) $. We set 
\begin{align*}
V(\mat{\mu}, \mat{\lambda}, \Sigma) &=& &\left( \bigcap_{i\in \Sigma}V_i^+(\mat{\mu}, \mat{\lambda})\right)/\sum_{j\in {\rm Int} (2{\mat{\mu}})\setminus\Sigma}\left( \bigcap_{k \in \Sigma \cup \{ j\}}V_k^+(\mat{\mu}, \mat{\lambda}) \right),\\
V(\mat{\mu}, \mat{\lambda}, \emptyset) &=& &V(\mat{\mu}, \mat{\lambda})/\Big(\sum\limits_{j\in{\rm Int} (2{\mat{\mu}})}V_j^{+}(\mat{\mu}, \mat{\lambda})\Big),\\
V(\mat{\mu}, \mat{\lambda}, {\rm Int} (2{\mat{\mu}})) &=& &\bigcap\limits_{i\in {\rm Int} (2{\mat{\mu}})}V_i^{+}(\mat{\mu}, \mat{\lambda}).
\end{align*}
\end{definition}

We will prove that the modules $V(\mat{\mu}, \mat{\lambda}, \Sigma) $ are simple, see Theorem \ref{thm-s-v}. Theorem \ref{module structure} implies the following.

\begin{lemma} \label{lem-c-ch}The module $V(\mat{\mu}, \mat{\lambda}, \Sigma)$  has central character $\chi_{\smat{\lambda}+ \smat{1}\mat{}}$ and cone $\mathcal C_{\Sigma, \Sigma'}$ for $\Sigma' = {\rm Int} (2{\mat{\mu}})\setminus \Sigma$. 
\end{lemma}

\begin{proof} The cone property follow from Lemma \ref{properties of cones} and the definition of  $V(\mat{\mu}, \mat{\lambda}, \Sigma)$.   For the central character statement we use Theorem \ref{module structure}. \end{proof}

Combining Proposition \ref{support}, Lemma \ref{essential support}, Lemma \ref{lem-c-ch}, and Proposition \ref{prop-class-sp} we obtain:
\begin{theorem} \label{thm-s-v} With the notation of Proposition \ref{prop-class-sp}, we have
$$
S(\chi_{\smat{\lambda} + \smat{1}\mat{}}, 2{\mat{\mu}} + \mat{\lambda} + \mat{1} + \mathcal Q_C,  \Sigma) \simeq V({\mat{\mu}}, \mat{\lambda}, \Sigma).
$$
 In particular, $L_C(\mat{\lambda} + \mat{1}) \simeq V(\mat{\frac{1}{2}}, \mat{\lambda}, \llbracket n \rrbracket )$ for even $n$, and  $L_C(\tau_1\mat{\lambda} + \mat{1}) \simeq V(\mat{\frac{1}{2}}, \mat{\lambda}, \llbracket n \rrbracket )$ for odd $n$. Furthermore, $T(W_{\smat{\lambda}\mat{}})$ is a highest weight vector of $V(\mat{\frac{1}{2}}, \mat{\lambda}, \llbracket n \rrbracket )$ (cf. Theorem \ref{HW tableau}).
 \end{theorem}

\begin{theorem} \label{thm-main} 
\begin{itemize}
\item[(i)] Every infinite-dimensional simple bounded $\mathfrak{sp}(2n)$-module is isomorphic to $V(\mat{\mu}, \mat{\lambda}, \Sigma)$ for some $\mat{\mu}\in\left(\mathbb{C}\setminus\mathbb{Z}\right)^{n}$, a   dominant $\mathfrak{so}(2n)$-weight $\mat{\lambda}\in(\frac{1}{2}+\mathbb{Z})^{n}$, and a subset $\Sigma$ of $ {\rm Int} (2{\mat{\mu}})$. 
\item[(ii)] $V(\mat{\mu}, \mat{\lambda}, \Sigma) \simeq V(\mat{\mu'}, \mat{\lambda'}, \Sigma')$  if and only if $\Sigma = \Sigma'$, and 
\begin{itemize}
\item[(a)]  $\mat{\lambda}=\mat{\lambda'}$ and $2(\mat{\mu}-\mat{\mu'})\in \mathcal Q_C$, or 
\item[(b)] $\mat{\lambda} = \tau_{1} (\mat{\lambda'})$ and $2(\mat{\mu} - \mat{\mu'}) - \epsilon_1 \in \mathcal Q_C$. 
\end{itemize}
\end{itemize}
\end{theorem}
\begin{proof}
By Theorem \ref{thm-s-v} and Proposition  \ref{prop-class-sp} we have $V({\mat{\mu}}, \mat{\lambda}, \Sigma) \simeq V({\mat{\mu'}}, \mat{\lambda'}, \Sigma')$ if and only if $\Sigma = \Sigma'$ and
\begin{itemize}
\item[(A)] $\mat{\lambda}=\mat{\lambda'}$ and $2(\mu-\mu')\in\mathbb{Z}^{n}$, or
\item[(B)] $\mat{\lambda'} + \mat{1} + \rho_C = \tau_{1} (\mat{\lambda} + \mat{1} + \rho_C)$ and $\mat{\lambda'} + 2\mat{\mu'} - \left(\mat{\lambda} + 2\mat{\mu}\right) \in \mathcal Q_C$.
\end{itemize}
Note that we can rewrite (B) in a simpler way, namely as  (ii)(b) in the theorem. 
\end{proof}

\begin{proposition} \label{prop-tw-loc}
$V(\mat{\mu}, \mat{\lambda}) \simeq D_{\llbracket n \rrbracket }^{\mat{\mu}} L_C(\mat{\lambda} + \mat{1})$.
\end{proposition}
\begin{proof}
We have that $F_{k,-k}$ acts injectively on $V(\mat{\mu}, \mat{\lambda})$ by \eqref{Formulas sp1}. By Proposition \ref{number of D-standard tableaux} all weight spaces of $V(\mat{\mu}, \mat{\lambda})$ have the same dimension, hence,  $F_{k,-k}$ acts bijectively on $V(\mat{\mu}, \mat{\lambda})$.  
Thus we have:
\begin{eqnarray*}
V(\mat{\mu}, \mat{\lambda}) & \simeq & D_{\llbracket n \rrbracket } V(\mat{\mu}, \mat{\lambda}) \simeq D_{\llbracket n \rrbracket } V(\mat{\mu}, \mat{\lambda},  {\rm Int} (2{\mat{\mu}}))  \\ & \simeq &  D_{\llbracket n \rrbracket } S(\chi_{\smat{\lambda} + \smat{1}\mat{}}, 2{\mat{\mu}} + \mat{\lambda} + \mat{1} + \mathcal Q_C, {\rm Int} (2{\mat{\mu}}))  \simeq D_{\llbracket n \rrbracket }^{\mat{\mu}} L_C(\mat{\lambda} + \mat{1}).
\end{eqnarray*}
The second isomorphism follows from the exactness of $ D_{\llbracket n \rrbracket }$, and the fact that $D_{\llbracket n \rrbracket } \left(V(\mat{\mu}, \mat{\lambda})/V(\mat{\mu}, \mat{\lambda},  {\rm Int} (2{\mat{\mu}}))  \right) =0$.
The third isomorphism follows from  Theorem \ref{thm-s-v}. The last isomorphism follows from Proposition \ref{sp-loc-inj}.
\end{proof}

\begin{corollary}\label{cor-inj}
The module $V(\mat{\mu}, \mat{\lambda})$ is the injective envelope of $V(\mat{\mu}, \mat{\lambda},{\rm Int} (2{\mat{\mu}}))$ in the category of bounded $\mathfrak{sp}(2n)$-modules. Furthermore,  $V(\mat{\mu}, \mat{\lambda}) \simeq  V(\mat{\mu'}, \mat{\lambda'})$ if and only if $\mat{\mu}, \mat{\lambda}, \mat{\mu'}, \mat{\lambda'}$ satisfy conditions (ii)(a) or (ii)(b) of Theorem \ref{thm-main}.
\end{corollary}

\begin{remark}
Proposition \ref{prop-tw-loc} leads to an explicit description of all twisted $\llbracket n \rrbracket $-localizations of highest weight modules in the category of bounded modules with central character $\chi_{\smat{\lambda} + \smat{1}\mat{}}$. By Proposition \ref{prop-missing}, the one that is ``missing'' is  the module $D_{\llbracket n \rrbracket }^{\frac{1}{2}(\varepsilon_1+ {\bf 1})}  L(\mat{\lambda} + \mat{1})$, But this module is isomorphic to $D_{\llbracket n \rrbracket }^{\bf1/2} L(\tau_1 \mat{\lambda} + \mat{1}) \simeq V(\mat{\frac{1}{2}}, \tau_1 \mat{\lambda} )$. Also, the other injectives (and hence, projectives) in $\mathcal B$ are obtained via a twist of products of long reflections.
\end{remark}

\section*{Appendix: Proof of Theorem \ref{HW tableau}}

In this Appendix we prove that $F_{k-1,k}T(W_{\smat{\lambda}\mat{}})=0$ for any $k=2,\ldots,n$, and $F_{-j,j}T(W_{\smat{\lambda}\mat{}}) =0$ for any $j\in \llbracket n \rrbracket $. For this purpose we need two lemmas.

Recall the formulas \eqref{operator A},  \eqref{operator B},  \eqref{operator C},  \eqref{operator D}.

\begin{lemma}\label{some identities}
Set $1\leq k\leq n$, and let $T(L)$ be any tableau in $\mathcal{B}(\mat{\mu},\mat{\lambda})$ such that $\ell'_{k1}=\frac{1}{2}$, $\ell'_{kk}=\ell_{kk}$, and  $\ell_{k, i}=\ell'_{ki}=\ell_{k-1,i}$, for any $2\leq i\leq k-1$. Then:

\begin{itemize}
\item[(i)] $B_{ki}(L)=0$ for any $i\in \llbracket k \rrbracket $.
\item[(ii)] $B_{kr}(L-\delta^{(k-1,s)})=0$, for any $s\in \llbracket k \rrbracket $, and any $2\leq r\leq k-1$.
\item[(iii)] $B_{kt}(L+\delta^{(k,i)'}+\delta^{(k-1,j)})=0$, for any $1\leq t\neq i\leq k$.
\item[(vi)] $B_{kj}(L+\delta^{(k,j)'}+\delta^{(k-1,j)})=0$, for any $2\leq j\leq k-1$.
\item[(v)] $D_{kijm}(L)=0$ whenever $j\neq 1$, and $i\neq j$.
\end{itemize}
\end{lemma}
\begin{proof}
All statements follow by a direct computation. 
\end{proof}
\begin{lemma}\label{more identities}
Set $k\in \llbracket n \rrbracket $, and let $T(L)$ be as in Lemma \ref{some identities} such that $\ell_{k1}+\ell_{k-1,1}=1$. Then for any $m\in \llbracket k-1 \rrbracket $ we have 
 $$\sum\limits_{i=1}^{k}D_{ki1m}(L)B_{ki}(L+\delta^{(k,i)'}+\delta^{(k-1,1)})=0.$$
\end{lemma}
\begin{proof} 

First note that $D_{ki1m}(L)=A_{k-1,m}(L)C_{k1}(L)\prod\limits_{a\neq m}^{k-1}(l^2_{k-1,1}-l^{'2}_{k-1,a})\widetilde{D}_{ki1m}(L)$, where $
\widetilde{D}_{ki1m}(L)=A_{ki}(L)\prod\limits_{a\neq i}^{k}(l^2_{k-1,1}-l^{'2}_{k,a})=\prod\limits_{a\neq i}^{k}\frac{(\ell_{k-1,1}^2-l'^2_{k,a})}{(\ell'_{ka}-\ell'_{ki})}
$

$$
\widetilde{D}_{ki1m}(L)=\begin{cases}
\frac{\prod\limits_{a\neq 1}^{k-1}(\ell_{k-1,1}^2-l^2_{k,a})}{\prod\limits_{a\neq 1}^{k}(\ell_{ka}-\ell'_{k1})},& \text{ if } i=1,\\
\frac{(l^2_{k-1,1}-l'^{2}_{k1})}{\ell'_{k1}-\ell_{ki}}\left(\frac{\prod\limits_{a\neq 1,i}^{k-1}(\ell_{k-1,1}^2-l^2_{k,a})}{\prod\limits_{a\neq 1,i}^{k}(\ell_{ka}-\ell_{ki})}\right),& \text{ if } i\neq 1,
\end{cases}
$$

$$
B_{ki}(L+\delta^{(k,i)'}+\delta^{(k-1,1)})=\begin{cases}
-4(l^2_{k-1,1}-l'^{2}_{k1})\prod\limits_{a\neq 1}^{k-1}(\ell_{ka}-\ell'_{k1}-1),& \text{ if } i=1,\\
-4(l^{2}_{k-1,1}-l^{2}_{k,i})\prod\limits_{a\neq 1}^{k-1}(\ell_{ka}-\ell_{ki}-1),& \text{ if } i\neq 1.
\end{cases}
$$

After factoring out  
$$
-4A_{k-1,m}(L)C_{k1}(L)(l^2_{k-1,1}-l'^{2}_{k1})\prod\limits_{a\neq m}^{k-1}(l^2_{k-1,1}-l^{'2}_{k-1,a})\prod\limits_{a\neq 1}^{k-1}(\ell_{k-1,1}^2-l^2_{k,a})
$$ 
from the left hand side, we see that it is enough to show:

\begin{align}\label{lagrange}
\small \frac{\prod\limits_{a\neq 1}^{k-1}(\ell_{ka}-\ell'_{k1}-1)}{\prod\limits_{a\neq 1}^{k}(\ell_{ka}-\ell'_{k1})}+\sum\limits_{i=2}^{k}\frac{\prod\limits_{a\neq 1}^{k-1}(\ell_{ka}-\ell_{ki}-1)}{(\ell'_{k1}-\ell_{ki})\prod\limits_{a\neq 1,i}^{k}(\ell_{ka}-\ell_{ki})}=\sum\limits_{i=1}^{k}\frac{\prod\limits_{a\neq 1}^{k-1}(c_{a}-c_{i}-1)}{\prod\limits_{a\neq i}^{k}(c_{a}-c_{i})},
\end{align}
where $c_{1}=\ell'_{k1}, c_{2}=\ell_{k2},\ldots,c_{k}=\ell_{kk}$. Note that $c_{i}\neq c_{j}$ for $i\neq j$. But with the aid of Lagrange interpolation formula, the polynomial $f(x)=\prod\limits_{a\neq 1}^{k-1}(c_{a}-x-1)$ of degree $k-2$ can be written as follows:

\begin{align}\label{Lagrange2}
f(x)=\sum\limits_{i=1}^{k}f(c_{i})\left(\frac{\prod\limits_{a\neq i}^{k}(c_{a}-x)}{\prod\limits_{a\neq i}^{k}(c_{a}-c_{i})}\right)=\sum\limits_{i=1}^{k}\left(\frac{\prod\limits_{a\neq 1}^{k-1}(c_{a}-c_{i}-1)}{\prod\limits_{a\neq i}^{k}(c_{a}-c_{i})}\right)\prod\limits_{a\neq i}^{k}(c_{a}-x).
\end{align}

Since $\prod\limits_{a\neq i}^{k}(c_{a}-x)$ is a polynomial of degree $k-1$ whose leading coefficient equals $(-1)^{k-1}$, by comparing the  coefficients of $x^{k-1}$ in (\ref{Lagrange2}) we obtain $$0=(-1)^{k-1}\sum\limits_{i=1}^{k}\left(\frac{\prod\limits_{a\neq 1}^{k-1}(c_{a}-c_{i}-1)}{\prod\limits_{a\neq i}^{k}(c_{a}-c_{i})}\right).$$\end{proof}

We are now ready to prove the theorem. We first note that $T(W_{\smat{\lambda}\mat{}})$ satisfies the conditions of Lemma \ref{some identities}, and Lemma \ref{more identities}. In particular:

\begin{enumerate}[$(1)$]
\item $F_{-k,k}(T(W))=\sum\limits_{i=1}^{k}B_{ki}(W)T(W-\delta^{(ki)'})=0$, by Lemma \ref{some identities}(i).
\item $2F_{k-1,k}(T(W))=\left[F_{k-1,-k},\ F_{-k,k}\right](T(W))=F_{-k,k}(F_{k-1,-k}(T(W)))$ by $(1)$.
\medskip 

We note that $B_{ki}(L)$ does not depend in row $(k-1)'$. Hence:

\begin{align*}
& F_{-k,k}(F_{k-1,-k}(T(W))) = \sum\limits_{i=1}^{k-1}C_{k,i}(W)\sum\limits_{r=1}^{k}B_{kr}(W-\delta^{(k-1,i)})T\left(W-\delta^{(k-1,i)}-\delta^{(kr)'}\right)+\\
&\sum\limits_{i,t=1}^{k} \sum\limits_{j,m=1}^{k-1}D_{k,i,j,m}(W) B_{kt}(W+\delta^{(k,i)'}+\delta^{(k-1,j)})T\left(W+\delta^{(k,i)'}+\delta^{(k-1,j)}+\delta^{(k-1,m)'}-\delta^{(k,t)'}\right)
\end{align*}
\begin{align*}
\stackrel{(a)}{=}&\sum\limits_{i,t=1}^{k} \sum\limits_{j,m=1}^{k-1}D_{k,i,j,m}(W) B_{kt}(W+\delta^{(k,i)'}+\delta^{(k-1,j)})T\left(W+\delta^{(k,i)'}+\delta^{(k-1,j)}+\delta^{(k-1,m)'}-\delta^{(k,t)'}\right)\\
\stackrel{(b)}{=}&\sum\limits_{i=1}^{k} \sum\limits_{j,m=1}^{k-1}D_{k,i,j,m}(W) B_{ki}(W+\delta^{(k,i)'}+\delta^{(k-1,j)})T\left(W+\delta^{(k-1,j)}+\delta^{(k-1,m)'}\right)
\end{align*}
\begin{align*}
=&\sum\limits_{i=1}^{k} \sum\limits_{m=1}^{k-1}D_{k,i,1,m}(W) B_{ki}(W+\delta^{(k,i)'}+\delta^{(k-1,1)})T\left(W+\delta^{(k-1,1)}+\delta^{(k-1,m)'}\right)+\\
&\sum\limits_{i=1}^{k} \sum\limits_{m=1,j=2}^{k-1}D_{k,i,j,m}(W) B_{ki}(W+\delta^{(k,i)'}+\delta^{(k-1,j)})T\left(W+\delta^{(k-1,j)}+\delta^{(k-1,m)'}\right)
\end{align*}
\begin{align*}
\stackrel{(c)}{=}&\sum\limits_{i=1}^{k} \sum\limits_{m=1}^{k-1}D_{k,i,1,m}(W) B_{ki}(W+\delta^{(k,i)'}+\delta^{(k-1,1)})T\left(W+\delta^{(k-1,1)}+\delta^{(k-1,m)'}\right)+\\
& \sum\limits_{m=1,j=2}^{k-1}D_{k,j,j,m}(W) B_{kj}(W+\delta^{(k,j)'}+\delta^{(k-1,j)})T\left(W+\delta^{(k-1,j)}+\delta^{(k-1,m)'}\right)
\end{align*}
\begin{align*}
\stackrel{(d)}{=}&\sum\limits_{i=1}^{k} \sum\limits_{m=1}^{k-1}D_{k,i,1,m}(W) B_{ki}(W+\delta^{(k,i)'}+\delta^{(k-1,1)})T\left(W+\delta^{(k-1,1)}+\delta^{(k-1,m)'}\right)
\end{align*}
\begin{align*}
=&\sum\limits_{m=1}^{k-1}\left(\sum\limits_{i=1}^{k} D_{k,i,1,m}(W) B_{ki}(W+\delta^{(k,i)'}+\delta^{(k-1,1)})\right)
T\left(W+\delta^{(k-1,1)}+\delta^{(k-1,m)'}\right)\\
\stackrel{(e)}{=}&0.
\end{align*}
 \end{enumerate}
In the identities above $(a)$ follows from Lemma \ref{some identities}(ii), $(b)$ follows from Lemma \ref{some identities}(iii), $(c)$ follows from Lemma \ref{some identities}(iv), $(d)$ follows from Lemma \ref{some identities}(v), and $(e)$ follows from Lemma \ref{more identities}.\end{proof}

\end{document}